\documentclass[a4paper]{article}

\usepackage{comment}
\usepackage{fullpage} 
\usepackage{parskip} 
\usepackage{tikz} 
\usepackage{hyperref}
\usepackage{xcolor}
\usepackage{soul}
\usepackage{amsmath,amssymb}
\usepackage{amsthm}
\usepackage{algorithm}
\usepackage{subcaption}
\usepackage{algpseudocode}
\newtheorem{theorem}{Theorem}[section]
\newtheorem{definition}[theorem]{Definition}
\newtheorem{lemma}[theorem]{Lemma}
\newtheorem{proposition}[theorem]{Proposition}

\newtheorem{assumption}{Assumption}
\usepackage{multirow}
\usepackage{multicol}
\usepackage{booktabs}
\usepackage{array}

\title{Enhancing Model Based Derivative Free Optimization using Direct Search}
\author{Zijun Li and Aswin Kannan\footnote{The first author is reachable at zijun.li@hu-berlin.de. The first author is a PhD candidate at Humboldt Universitaet zu Berlin, Germany under the guidance of the second author. The second author is the corresponding author and is reachable at both aswin.kannan@hu-berlin.de and aswin.kannan@iiitb.ac.in. The ORCID-ID for the corresponding author is 0000-0002-4698-7678. \\ 
The second author was a faculty member and Junior Research Group Leader at Humboldt Universitaet zu Berlin till December 2024. He is currently a faculty member at the International Institute of Information Technology, Bangalore, India (from January 2025). He still retains an affiliation to his Junior Group in Berlin.}}

\begin{document}
\maketitle
\abstract{
We consider single and multiobjective simulation-based optimization problems. Simulation-based optimization has traditionally used both model-based and search-based methods, often in isolation. Model-based methods include trust region approaches and Bayesian optimization, while search methods include genetic algorithms and Direct Search techniques. In this work, we propose a switching framework that leverages Direct Search methods to enhance the performance of model-based optimizers. Our contributions are twofold. First, in the single-objective setting, we analyze and prove asymptotic convergence of the proposed switching approach. Second, motivated by applications in machine learning, we consider classification and regression problems, where the objectives span accuracy, computational time, algorithmic bias, and sparsity. The models range from neural networks and decision trees to simpler KNN baselines. For machine learning, we introduce a warm-starting mechanism—reusing weights from previous hyperparameter or architectural configurations—to exploit problem structure and accelerate training. Beyond ML tasks, we evaluate the method on standard CUTEr test problems and compare its performance against Bayesian and trust region solvers. We observe consistently strong numerical performance, suggesting promise for the proposed switching-based approach.
}
\section{Introduction}

Derivative-Free Optimization (DFO) is an important branch of mathematical optimization focused on solving problems where the derivative information of the objective function may be unknown, noisy, expensive to obtain, or unavailable due to proprietary reasons. Post the inception of BOBYQA 
and Powell's works, these have gained significant prominence in a multitude of applications in engineering, science, and medicine. 

Many engineering applications, such as groundwater management, automotive and aircraft design, and chemical process optimization involve complex simulation-based optimization problems. In groundwater modeling, for example, the governing dynamics are typically described by partial or ordinary differential equations, and the associated simulations are computationally intensive. Constraints often pertain to physical or regulatory limits on flow levels, while objective functions may reflect economic costs such as those incurred by well installation or pumping operations. In automotive design, decision variables might include gear ratios, tire radii, or engine specifications, with objectives such as maximizing fuel economy (e.g., miles per gallon), subject to constraints like acceleration thresholds (e.g., $0–60$ mph times). Similar formulations arise in chemical process design, such as styrene production, and have been studied extensively across engineering domains~\cite{audet2008nonsmooth,SASW15ice}. In the context of machine learning, derivative-free optimization (DFO) plays a critical role in both hyperparameter tuning and in bi-level optimization problems, where the outer-level problem depends on solutions of inner learning tasks. The widespread reliance on surrogate losses—often dictated by solver artifacts—poses additional challenges, especially when optimization must align with user-specific or application-driven metrics. These difficulties are further compounded in multi-objective settings, where gradient-based optimization becomes even less tractable, motivating the use of DFO methods~\cite{golovin2017google,karasozen2007survey}.

While finite-difference approximations are a classical remedy for gradient unavailability, they become increasingly impractical when function evaluations are expensive, particularly in high-dimensional settings. Automatic differentiation provides an alternative in some cases, but is often infeasible for proprietary or multi-architecture simulation environments. Consequently, DFO methods are not merely convenient but necessary. Broadly, DFO algorithms can be categorized into model-based methods and search-based techniques. Model-based methods include both Bayesian optimization (BO) and trust-region (TR) frameworks. These approaches construct surrogate models of the objective function using only function evaluations. TR methods typically rely on quadratic interpolation or regression, while BO methods employ Gaussian process priors to build probabilistic models. Both strategies offer strong theoretical guarantees, including global convergence to first-order stationary points, and under suitable assumptions, to second-order points as well~\cite{larson2019derivative}. Recent advancements have also extended TR methods to handle noisy or stochastic function evaluations through adaptive sampling and regression-based model construction~\cite{gao2021performance}.

Search-based methods fall into two primary categories: heuristic algorithms, such as genetic algorithms~\cite{deb2002fast}, and direct search (DS) techniques~\cite{audet2006mesh,bigeon2021dmulti}. Genetic algorithms are population-based and rely on stochastic evolution principles, while DS methods explore the decision space via deterministic or randomized patterns—often iterating over a sequence of directions that collectively span $\mathbb{R}^n$. A classical example includes coordinate search, wherein directional evaluations are conducted until sufficient descent is obtained. Further methodological details and algorithmic formulations are presented in the subsequent sections of this paper.
\\ \hspace{1mm} \\

\paragraph{\textbf{Our Key Motivation:}}  
Model-based DFO methods, while appealing for their efficient use of function evaluations, can be computationally demanding due to the cost of constructing and minimizing surrogate models. Their performance is often contingent on favorable structural properties of the objective function. Trust-region (TR) methods, in particular, are sensitive to model quality—poorly fitted surrogates may result in stagnation or premature termination. As with many local search strategies, TR methods are prone to convergence to suboptimal local minima, especially when initial trust regions are poorly scaled. Moreover, their scalability is limited in higher dimensions unless exploitable structure, such as partial separability, is available to ease model construction~\cite{karasozen2007survey}. 

Bayesian optimization (BO) offers a more global perspective, but its capacity for local refinement is typically weaker than that of search-based methods. In contrast, direct search (DS) techniques circumvent explicit modeling by exploring the decision space through function evaluations at carefully chosen points. Canonical methods in this class include Generalized Pattern Search (GPS)~\cite{torczon1997convergence}, Mesh Adaptive Direct Search (MADS)~\cite{audet2006mesh}, and the Nelder--Mead simplex method~\cite{nelder1965simplex}. Their conceptual simplicity, lack of smoothness requirements, and established convergence theory make them particularly attractive in nonsmooth and modestly sized settings.

DS methods also offer practical advantages in warm-start scenarios and low-dimensional problems—features especially relevant to our setting, where the dimension is small ($n < 25$), and warm-starts can be naturally invoked via prior evaluations in ML-based inner problems. While DS methods are generally robust, they often converge slowly and may require a large number of evaluations to attain high-accuracy solutions, particularly as dimensionality increases.
Our central objective is to harness the complementary strengths of both model-based and direct search (DS) methods, motivating the development of hybrid approaches. Model-based methods are typically more efficient on smooth problems, often requiring fewer function evaluations when the underlying structure is favorable. In contrast, DS methods tend to perform better in the presence of noise, nonsmoothness, or local irregularities, where modeling becomes unreliable or expensive. By combining these paradigms, one can aim to balance global efficiency with local robustness, particularly in settings where problem characteristics vary across the search space.

\paragraph{\textbf{Contributions:}}  
The primary contribution of this work is the development of a novel derivative-free optimization (DFO) algorithm that dynamically switches between model-based and direct search (DS) methods during optimization. The goal is to leverage the complementary strengths of both paradigms to achieve more reliable and efficient performance than either method in isolation. Specifically, we provide theoretical guarantees for a hybrid approach combining trust-region (TR) and DS methods; the integration of Bayesian optimization (BO) with DS remains an open direction for future research.

While our algorithm is tested on general benchmark problems, its design is particularly tailored to multiobjective machine learning settings, where varying smoothness and noise properties often demand adaptive strategies. Our contributions are summarized below:
\begin{itemize}
    \item We present convergence guarantees for the proposed hybrid TR–DS method in the single-objective setting. Extensions to multiobjective optimization remain an avenue for future exploration.
    \item We evaluate our algorithm on a suite of problems, including multiobjective machine learning tasks from the energy domain, as well as standard problems from the CUTEst test set.
\end{itemize}

To clarify our proposed contribution, we provide a comparison of existing DFO strategies based on their smoothness assumptions, model requirements, and convergence properties. Table~\ref{tab:comparing_dfo} provides a comparative overview of classical TR methods, DS techniques (e.g., GPS and MADS), Bayesian optimization, and our proposed TR-DS method. While model-based methods offer efficiency on smooth problems and direct search provides robustness in nonsmooth settings, the TR-DS framework we propose aims to leverage their strengths by providing a dynamic switching mechanism.

\begin{table}[h]
\centering
\caption{
Comparison of DFO methodologies and the proposed TR-DS method.}
\label{tab:comparing_dfo}
\small
\begin{tabular}{p{2.1cm}|p{2.4cm}|p{2.1cm}|p{2.1cm}|p{3.8cm}}
\hline
\textbf{Methodology} & \textbf{Smoothness Assumptions} & \textbf{Model Requirements} & \textbf{Convergence Guarantee} & \textbf{Key Distinction / Strengths} \\
\hline
Classical TR-based (e.g., BOBYQA) &
$C^1$ + $L$-smooth &
Fully-linear models on trust regions &
First-order stationarity &
High efficiency on smooth problems; can stagnate when models are poor or regions are ill-scaled. \\
\hline
Generalized Pattern Search (GPS) &
None / local Lipschitz (for nonsmooth theory) &
None &
Clarke stationarity &
Simple and parallelizable; robust to noise but fixed directions can limit convergence speed. \\
\hline
Mesh Adaptive Direct Search (MADS) &
None &
None &
Clarke stationary &
Strong nonsmooth theory; asymptotically dense directions ensure robustness but requires many evaluations. \\
\hline
Bayesian Optimization (e.g., EGO) &
Kernel-dependent (typically continuous) &
Probabilistic surrogate (Gaussian Process) &
Asymptotic global convergence under strong assumptions &
Sample efficient for expensive functions; good global exploration but scales poorly with dimensions ($d > 25$). \\
\hline
\textbf{Proposed TR-DS} &
\textbf{$C^1$ + Lipschitz continuity} &
\textbf{Fully-linear models on trust regions} &
\textbf{First-order stationarity} &
\textbf{Dynamic switching mechanism: Monitors $\rho_k$ to invoke DS specifically for escaping model stagnation, combining TR speed with DS robustness.} \\
\hline
\end{tabular}
\end{table}


The remainder of the paper is structured as follows. In Section~\ref{sec:preliminary}, we formalize the optimization problem and introduce essential definitions. Section~\ref{sec:algorithm} describes our proposed hybrid TR–DS algorithm in detail, including both the trust-region and direct search components, along with the switching mechanism. This section also outlines a variant involving the combination of BO and DS. In Section~\ref{sec:convergence_theory}, we present the convergence analysis and establish theoretical guarantees for the TR–DS method. Section~\ref{sec:numeric} reports numerical experiments comparing the proposed approach with other DFO solvers on benchmark and real-world problems, with a focus on multiobjective machine learning tasks. Finally, Section~\ref{sec:conclusion} concludes the paper and discusses directions for future research.


\begin{table}[h]
\centering
\caption{Summary of frequently used symbols}
\label{tab:notations}
\begin{tabular}{|c|p{9cm}|}
\hline
Symbol & Description \\
\hline
$dp_s$ & data profile \\
$\delta_k$ & mesh size of the $k$th iteration of the DS method \\
$\Delta_k$ & trust-region radius of the $k$th iteration \\
$\Delta_{\min}$ & minimum trust-region radius \\
$\Delta_{\max}$ & maximum trust-region radius \\
$\mathbf{g}$ & gradient \\
$H$ & Hessian \\
$K_{\mathrm{ds}}$ & maximum number of evaluations occurs per DS iteration \\
$l_i$ & the $i$th Lagrange polynomial basis \\ 
$N^{ds}$ & the current number of TR-to-DS switch \\
$N^{ds}_{\text{fail}}$ & the current counts of failed improvement in the DS method \\
$m$ & constructed model \\
$M(\Phi, Y)$ & interpolation matrix \\
$\Phi$ & polynomial basis \\
$\mathcal{B}$ & trust region \\
$r_{p,s}$ & performance ratio \\
$\rho$ & acceptance ratio \\
$\varrho_s$ & performance profile \\
$\tau$ & tolerance \\
$T^{ds}_{\max}$ & maximum number of TR-to-DS switch \\
$T^{ds}_{\text{fail}}$ & total counts of failed improvement in the DS method \\
$w_i$ & the weight of the $i$ objective \\
$Y$ & interpolation set \\
\hline
\end{tabular}
\end{table}

\section{Problem setup and basic definitions} \label{sec:preliminary}
In this work, we are concerned with the following unconstrained optimization problem:
\begin{align*}
    \min_{\mathbf{x} \in \mathbb{R}^n} f(\mathbf{x}),
\end{align*}
where \( f: \mathbb{R}^n \to \mathbb{R} \) denotes the objective function. We assume that the explicit form of \( f \) is unknown and that its derivatives are either inaccessible or computationally prohibitive to evaluate.

Our primary aim is to integrate the strengths of model-based methods (MBMs) with the flexibility of direct search strategies. It is important to emphasize that our focus explicitly excludes evolutionary approaches such as genetic algorithms; rather, we confine our attention to derivative-free direct search techniques.

Among the various domains of application, our central motivation arises from multiobjective machine learning, a field that has seen increasing prominence due to its foundational relevance in contemporary data-driven problems. To that end, we provide a concise overview of multiobjective optimization and the notion of Pareto frontiers, in the interest of both completeness and clarity.

In addition, we briefly introduce key conceptual underpinnings of Bayesian Optimization (BO), Trust Region (TR) methods, and Direct Search (DS) algorithms. While our theoretical contributions pertain specifically to a hybrid approach combining TR and DS, we include a discussion of BO for context and completeness. Though both BO and TR fall within the class of model-based strategies, their operational philosophies diverge significantly: BO adopts a global perspective and typically accommodates noisy or expensive evaluations, whereas TR methods are inherently local and presuppose a degree of smoothness in the objective function.

\subsection{Trust Region Methods}
Among model-based methods in derivative-free optimization (DFO), trust region (TR) schemes remain the most widely adopted. While their origins trace back to foundational developments in mathematical programming~\cite{conn2000trust}, their deployment in simulation-based optimization gained momentum in the 1990s, particularly following the influential work of Powell~\cite{powell2006newuoa}. The central idea behind TR methods is rooted in the assumption that a smooth objective function can be approximated locally by a quadratic model within a restricted domain, known as the trust region. This region is typically defined by a radius around the current iterate, within which the model is assumed to be a reliable proxy for the objective.

In classical TR methods, this quadratic model is derived using gradient and Hessian information. However, in model-based DFO settings, where derivatives are unavailable or unreliable, surrogate models are instead constructed using interpolation over a set of sampled points. These models aim to mimic local curvature while relying on relatively few evaluations of the objective function. A basic TR algorithm~\cite{conn2009introduction, larson2019derivative} at iteration $k$ constructs a model $m_k(\mathbf{x})$ around the current point $\mathbf{x}_k$ and solves the following subproblem $\mathbf{s}_k$ within a ball of radius $\mathcal{B}(\mathbf{x}_k,\Delta_k)$:
\begin{align*}
\mathcal{B}(\mathbf{x}_k, \Delta_k) = \left\{\mathbf{x} \in \mathbb{R}^n : |\mathbf{x} - \mathbf{x}_k| \leq \Delta_k \right\}, \quad
\mathbf{s}_k = \underset{|\mathbf{s}|\leq \Delta_k}{\arg \min} \hspace{1mm} m_k(\mathbf{s}). 
\end{align*}

The model $m_k$ is typically quadratic parametrized by $c$, $\mathbf{g}$, and $H$. Note that $\mathbf{g}_k$ and $H_k$ approximate the gradient and Hessian of the unknown function $f(\mathbf x)$. These can be written in any of the two forms stated as follows.
\begin{align*}
m_k(\mathbf{x}_k + \mathbf{s}) = m_k(\mathbf{x}_k) + \mathbf{g}_k^T \mathbf{s} + \frac{1}{2} \mathbf{s}^T H_k \mathbf{s}, \hspace{1mm} \text{or} \hspace{1mm} \nonumber \\
m_k(\mathbf{x}) = c + \mathbf{g}^T \mathbf{x} + \frac{1}{2}\mathbf{x}^T H \mathbf{x}, \hspace{1mm} \text{where} \hspace{1mm} \mathbf{x} = \mathbf{x}_k+\mathbf{s}.
\end{align*}
Given samples and evaluations $\left\{\mathbf{x}_i, f(\mathbf{x}_i) \right\}_{i=1}^{m}$, the interpolation problem can be written as follows. 
\begin{align*}
& \min_{c, g, H} \hspace{3mm} \|H - \bar{H}\|_F^2 \nonumber \\
& \text{subject to: } c + \mathbf{g}^T \mathbf{x}_i + \frac{1}{2} (\mathbf{x}_i)^T H \mathbf{x}_i = f(\mathbf{x}_i).
\end{align*}
Here, the objective is to ensure that the current Hessian remains close to a reference $\bar{H}$ from the previous iteration, while satisfying interpolation constraints~\cite{powell2006newuoa}. In noisy regimes, some approaches forgo curvature matching and instead use simpler linear models by reducing the norm of the hessian in an absolute sense~\cite{SMW08MNH}. 
The quality of the step is then evaluated using the ratio
\begin{align}
\rho_k = \frac{f(\mathbf{x}_k) - f(\mathbf{x}_k + \mathbf{s}_k)}{m_k(\mathbf{x}_k) - m_k(\mathbf{x}_k + \mathbf{s}_k)}, \label{eq:rho_ratio}
\end{align}
which compares actual progress to predicted progress. Based on $\rho_k$, the trust region radius $\Delta_k$ is adjusted, and the iterate is either accepted or rejected. Crucially, these methods do not require precise gradients—descent can still be achieved using models that are fully linear or fully quadratic, concepts that relate to the approximation quality over the trust region. These ideas are further developed in sections~\ref{sec:algorithm} and~\ref{sec:convergence_theory}. We present the following definition of fully linear models, which marks the accuracy of the model deployed. 
\begin{definition} (Definition $10.3$ in~\cite{conn2009introduction}). \label{assumption:def_fully_linear}
    Given $\Omega \in \mathbb{R}^n$ is bounded and a function $f$ satisfies Assumption \ref{ass:smoothness}. A model $m_k(\mathbf{x} + \mathbf{s})$ for the function $f(\mathbf{x} + \mathbf{s})$ is called fully linear in $\mathcal{B}(\mathbf{x}, \Delta)$ if model $m \in C^1$ and there exist constants $\kappa_{ef}, \kappa_{eg}>0$, the following holds: 
    \begin{itemize}
        \item the error between the model and the function holds
        \begin{align*} 
        |f(\mathbf{x} + \mathbf{s}) - m (\mathbf{x} + \mathbf{s})| \leq \kappa_{ef} \Delta^2,
        \end{align*}
        and 
        \item the error between the gradient of $m_k$ and the gradient of $f$ holds
        \begin{align} \label{eq:FL_error_gradient}
        ||\nabla f(\mathbf{x} + \mathbf{s}) - \nabla m (\mathbf{x} + \mathbf{s})|| \leq \kappa_{eg} \Delta.
        \end{align}
    \end{itemize}
    for all $||\mathbf{s}||\leq \Delta$.
\end{definition}

\subsection{Direct Search}

Direct Search (DS) methods are useful when the objective function is non-smooth, noisy, or when derivatives are unavailable or unreliable. Unlike gradient-based methods or model-based approaches, DS algorithms only rely on evaluating the objective function at carefully chosen points, using a set of directions that span the space. Some well-known DS methods include Generalized Pattern Search (GPS)~\cite{torczon1997convergence,audet2002analysis}, Mesh Adaptive Direct Search (MADS)~\cite{audet2006mesh, bigeon2021dmulti}, Generating Set Search (GSS), and the Nelder-Mead simplex method~\cite{nelder1965simplex, mckinnon1998convergence}. In this work, we focus on MADS, which is widely accepted and has strong convergence guarantees.

Each iteration of MADS consists of two main steps: a search step (optional, more global) and a poll step (mandatory, more local). Both steps evaluate the objective on a mesh constructed around the current best solution $\mathbf{x}_k$. The mesh is defined as:
\begin{align*}
    M_k = \{\mathbf{x}_k + \delta_k \mathcal{D}^k z : z \in \mathbb{R}^p\},
\end{align*}
where $\delta_k > 0$ is the current mesh size (step length), and $\mathcal{D}^k$ is a matrix containing directions that form a positive spanning set. A typical choice for $\mathcal{D}^k$ includes all unit coordinate directions and their negatives, forming a simple grid around $\mathbf{x}_k$.

\textbf{Search Step:}  
The algorithm evaluates $f(\mathbf{x})$ at a small number of points in $S_k \subset M_k$. If a better point is found (i.e., $f(t) < f(\mathbf{x}_k)$ for some $t \in S_k$), it becomes the new $\mathbf{x}_{k+1}$, and the mesh size is increased (to allow broader exploration). If the search step is successful, the poll step is skipped.

\textbf{Poll Step:}  
If the search step fails, the algorithm evaluates $f(\mathbf{x})$ at points in the poll set:
\begin{align*}
    \mathcal{P}_k = \{\mathbf{x}_k + \delta_k d : d \in \mathcal{D}^k\}.
\end{align*}
If any of these points improve the objective, the poll is declared successful, and the algorithm updates $\mathbf{x}_k$ and expands $\delta_k$. If not, the mesh is refined by shrinking $\delta_k$ (e.g., $\delta_{k+1} = \tau \delta_k$ for some $0<\tau<1$), allowing for finer search.

\textbf{Termination:}  
The algorithm stops when certain conditions are met—for instance, when $\delta_k$ becomes too small (below a threshold $\delta_{\min}$) or when the allowed number of function evaluations is exhausted. For more details and convergence results, we refer the reader to~\cite{audet2006mesh, bigeon2021dmulti}.

\subsection{Bayesian optimization}

Bayesian optimization (BO) is a global optimization technique designed for expensive black-box functions, especially when gradients are not available. 
BO works by building a probabilistic surrogate model—usually a Gaussian Process (GP)—to approximate the unknown function $f$ over a bounded domain. At each step, an acquisition function is used to decide the next evaluation point by balancing two goals: exploring uncertain areas of the function and exploiting areas already known to perform well. This balance allows BO to find near-optimal solutions using fewer function evaluations than DS methods.

BO is especially useful for problems with low to moderate dimensions and high evaluation costs, such as tuning hyperparameters in machine learning models. However, when the function is cheap to evaluate or the input space is very high-dimensional, BO can become computationally expensive. We summarize a simple and generic pseudocode for BO as follows. For more background and technical depth, the reader is referred to~\cite{frazier2018tutorial,belakaria2020uncertainty,kannan2021hyperaspo}. 
\begin{itemize}
    \item \textbf{Initialization:} Start with an initial dataset
    \[
    \mathcal{R}_0 = \left\{ \left(\mathbf{x}_i, f(\mathbf{x}_i) \right)\right\}_{i=1}^k,
    \]
    where $\mathbf{x}_i$ are input samples and $f(\mathbf{x}_i)$ are the corresponding function values.
    
    \item \textbf{Model Estimation:} Fit a surrogate model (typically a Gaussian Process) to estimate the mean $\mu(\mathbf{x})$ and standard deviation $\sigma(\mathbf{x})$ of the function over the domain, based on $\mathcal{R}_t$.

    \item \textbf{Acquisition Optimization:} At iteration $t+1$, select the next query point by solving
    \[
    \mathbf{x}_{t+1} = \arg\min_x \alpha\left(\mu(\mathbf{x}), \sigma(\mathbf{x})\right),
    \]
    where $\alpha$ is the acquisition function that trades off exploration and exploitation.

    \item \textbf{Update and Repeat:} Evaluate $f(\mathbf{x}_{t+1})$ and update the dataset:
    \[
    \mathcal{R}_{t+1} = \mathcal{R}_t \cup \left\{(\mathbf{x}_{t+1}, f(\mathbf{x}_{t+1}))\right\}.
    \]
    Repeat the process until a termination criterion is met (e.g., budget limit or convergence).
\end{itemize}

\subsection{Multi-objective optimization}

A multi-objective optimization problem (MOP) involves minimizing (or maximizing) $m \geq 2$ conflicting objective functions over a feasible decision space. Formally, it is defined as:
\begin{align*}
\min_{x \in \mathcal{X}} F(\mathbf x) = \left[ f_{1}(\mathbf x), \hdots, f_{m}(\mathbf x) \right],
\end{align*}
where $\mathbf x \in \mathbb{R}^n$ is the decision variable, $\mathcal{X} \subseteq \mathbb{R}^n$ is the feasible set, and $f_i:\mathcal{X} \to \mathbb{R}$ are the objective functions.

A distinctive aspect of MOPs is that the objectives often conflict—improving one may worsen another. For example, in car design, minimizing cost can compromise safety or fuel efficiency. Unlike single-objective problems, MOPs typically do not have a single optimal solution. Instead, optimality is defined using the notion of Pareto dominance.

\begin{definition}[Pareto Dominance] \label{def:MOO_dominate}
A point $\mathbf x^a$ is said to \textbf{dominate} $\mathbf x^b$ if:
\begin{align*}
f_j(\mathbf x^a) < f_j(\mathbf x^b) \text{ for at least one } j, \quad \text{and} \quad f_i(\mathbf x^a) \leq f_i(\mathbf x^b) \text{ for all } i \neq j.
\end{align*}
We denote this as $\mathbf x^a \succcurlyeq \mathbf x^b$.
\end{definition}

In other words, $x^a$ is no worse in all objectives and strictly better in at least one.

\begin{definition}[Pareto Optimality] \label{def:MOO_pareto_optimal}
A point $x^* \in \mathcal{X}$ is \textbf{Pareto optimal} if no other point in $\mathcal{X}$ dominates it. The set of all such non-dominated solutions is called the \textbf{Pareto set}, and their image under $F$ is known as the \textbf{Pareto front}.
\end{definition}
In practical MOPs, the aim is to approximate the Pareto front well. A decision maker can then select a preferred solution based on specific trade-offs—e.g., balancing accuracy and model complexity in machine learning, or cost and safety in engineering design.

\paragraph{Algorithms:} Approximations to the Pareto frontier are commonly obtained using scalarization techniques, $\epsilon$-constraint methods, or heuristics. In this work, we do not consider heuristics. Scalarization methods reduce an MOP to a sequence of single-objective problems by assigning weights to each objective. A widely used scalarization is the \textit{weighted sum method}~\cite{ehrgott2002multiobjective,kim2006adaptive}:
\begin{align*}
    F_{ws}(\mathbf x) = \sum_{i=1}^m w_i f_i(\mathbf x),
\end{align*}
where weights satisfy $w_i \geq 0$ for all $i = 1, \ldots, m$ and $\sum_{i=1}^m w_i = 1$. Each choice of weights typically gives a different point on the Pareto frontier. Another popular approach is the \textit{$\epsilon$-constraint method}~\cite{mavrotas2009effective,deb2016multi}, where one objective is minimized while the others are shifted to the constraint space. For example, in a two-objective problem where both $f_1(\mathbf x)$ and $f_2(\mathbf x)$ are normalized between $0$ and $1$, the problem takes the form:
\begin{align*}
    \min_x \ f_1(\mathbf x), \quad \text{subject to: } f_2(\mathbf x) \leq \epsilon.
\end{align*}
Here, $\epsilon$ is varied over a range (e.g., from 0.1 to 1.0) to produce different solutions. The $\epsilon$-constraint method is especially effective when the Pareto frontier is nonconvex, where weighted sum methods may not capture all trade-offs~\cite{steuer1983interactive,dachert2012augmented}. We emphasize that the focus of this work is not on comparing multi-objective optimization techniques, but on demonstrating the advantage of hybrid algorithms such as BO/DS or TR/DS when applied to MOPs. Hence, the question on comparing weighted methods or $\epsilon-$constraints or heuristics for MOPs become less relevant in our context. 
\section{Algorithm} \label{sec:algorithm}
Prior to formally defining the algorithm, we introduce some foundational concepts involving \textbf{Lagrange polynomials} and \(\Lambda\)\textbf{-poisedness}, which underpin DFO (Derivative-Free Optimization) methods within trust-region frameworks. Broadly, \(\Lambda\)-poisedness characterizes the \emph{quality and geometric configuration} of the interpolation set, relating to the \emph{conditioning} and \emph{linear independence} of selected points. Lagrange polynomials form a \emph{basis for interpolation}, enabling the construction of surrogate models. These polynomials play a central role in analyzing approximation quality and are instrumental in establishing error bounds (as discussed in Section~\ref{sec:convergence_theory}) between the objective function and its surrogate.

\paragraph{Lagrange Polynomials and \(\Lambda\)-Poisedness.}
The accuracy of a quadratic surrogate model depends on how well its sample points are distributed within the trust region. The notion of \(\Lambda\)-poisedness quantifies this distribution and is central to model-based DFO algorithms. Although standard in the literature, we formally present the relevant definitions for clarity and completeness.

\begin{definition}[Definitions 3.3 and 3.6 in~\cite{conn2009introduction}] \label{def:lagrange_poly}
Let \( Y = \{ y^0, \dots, y^p \} \subset \mathbb{R}^n \) be a set of interpolation points. The corresponding \textbf{Lagrange polynomials} \( l_j(\mathbf x) \), for \( j = 0, \dots, p \), are defined by:
\begin{align*}
    l_j(y^i) = \delta_{ij} =
    \begin{cases}
    1, & \text{if } i = j \\
    0, & \text{if } i \neq j
    \end{cases}
\end{align*}
The set \( Y \) is said to be \(\Lambda\)-poised in a trust region \( \mathcal{B} \) if the Lagrange polynomial basis \( \{ l_0, \dots, l_p \} \) satisfies:
\begin{align*}
    \max_{0 \leq j \leq p} \max_{\mathbf{x} \in \mathcal{B}} |l_j(\mathbf{x})| \leq \Lambda
\end{align*}
\end{definition}
Lower values of \( \Lambda \) are desirable, as they indicate well-spread interpolation points and lead to more robust surrogate models. In practice, we test whether the current set \( Y \) is \(\Lambda\)-poised within the trust region \( \mathcal{B} \). If the condition fails, the model is refined—typically by sampling additional points or updating the interpolation set—until \(\Lambda\)-poisedness is achieved.

\paragraph{Switching Method:}
We propose an adaptive optimization framework that dynamically integrates Trust Region (TR) and Direct Search (DS) methods. While TR methods are typically effective in local optimization, they can suffer from stagnation when the trust region radius becomes too small, particularly in derivative-free settings where gradient approximations may be unreliable or uninformative. In such cases, we transition to a DS strategy, which leverages a broader set of search directions and is capable of taking larger, exploratory steps. This can help the solver escape shallow local minima or regions of slow progress.

To ensure continuity in the optimization process, we retain the current trust region radius during this switch, interpreting it as the mesh size in the DS framework. We are, however, not overly rigid with respect to descent criteria in the DS phase, acknowledging that multiple exploratory steps can be computationally inexpensive in our setting—a property supported by the underlying problem structure, as will be elaborated later. Nonetheless, DS methods also have their limitations, particularly in the absence of gradient information, which makes it difficult to impose structure or adapt efficiently. Therefore, after a finite number of DS iterations, if sufficient progress is not observed, we revert back to the TR scheme. We note that this switching mechanism is also applicable in the context of Bayesian Optimization combined with Direct Search (BO-DS). Although we do not explore the BO-DS variant in detail here, we refer the reader to~\cite{lion24-aswin} for a comprehensive treatment of that framework.


\paragraph{Algorithm:} The technical steps behind our joint scheme are formally described in algorithm~\ref{alg:alg_1}. We expand on the fine aspects and explain in detail the related steps as follows.
\begin{itemize}
\item Step 1: \textbf{Initialization.} We initialize with iteration index $k = 0$, starting point $\mathbf{x}_0 \in \mathbb{R}^n$, and initial trust region (TR) radius $\Delta_0 \in \mathbb{R}$, bounded by $\Delta_{\min}$ and $\Delta_{\max}$. The acceptance thresholds are set as $1 > \eta_1 > \eta_0 > 0$, with TR update parameters $\gamma_{\text{inc}} > 1 > \gamma_{\text{dec}} > 0$. The initial model and point set are denoted by $m_0(\mathbf{x})$ and $Y_0$. For the Direct Search (DS) phase, we define a threshold $\epsilon_{ds}$, the maximum and current number of TR-to-DS switches as $T^{ds}_{\max}$ and $N^{ds}$, and the total and current counts of failed improvement attempts as $T^{ds}_{\text{fail}}$ and $N^{ds}_{\text{fail}}$. All counters are initialized with $N^{ds} = 0$ and $N^{ds}_{\text{fail}} = 0$.

\item Step 2: \textbf{Trust Region (TR) Phase.}
At iteration $k$, solve the trust region subproblem~\eqref{eq:subproblem_TR} to obtain a step $\mathbf{s}_k$. Evaluate $f(\mathbf{x}_k + \mathbf{s}_k)$ and compute $\rho_k$ as in~\eqref{eq:rho_ratio}. Based on $\rho_k$, update $\mathbf{x}_k$, $\Delta_k$, and the model $m_k$ as follows. Thereafter, set $k \leftarrow k+1$.
\begin{itemize}
    \item \textbf{Successful} ($\rho_k \geq \eta_1$):  
    $\mathbf{x}_{k+1} = \mathbf{x}_k + \mathbf{s}_k$,  
    $\Delta_{k+1} = \min\{\gamma_{\text{inc}} \Delta_k, \Delta_{\max}\}$,  
    update $Y_{k+1} = Y_k \cup \{\mathbf{x}_{k+1}\}$ and model $m_{k+1}$.
    \item \textbf{Acceptable} ($\eta_0 \leq \rho_k < \eta_1$, $m_k$ fully linear):  
    $\mathbf{x}_{k+1} = \mathbf{x}_k + \mathbf{s}_k$,  
    $\Delta_{k+1} = \gamma_{\text{dec}} \Delta_k$,  
    update $Y_{k+1} = Y_k \cup \{\mathbf{x}_{k+1}\}$ and $m_{k+1}$.
    \item \textbf{Model improvement} ($\rho_k < \eta_1$, $m_k$ not fully linear):  
    $\mathbf{x}_{k+1} = \mathbf{x}_k$,  
    $\Delta_{k+1} = \Delta_k$,  
    improve $m_k$ for full linearity.
    \item \textbf{Unsuccessful} ($\rho_k < \eta_0$, $m_k$ fully linear):  
    $\mathbf{x}_{k+1} = \mathbf{x}_k$,  
    $\Delta_{k+1} = \gamma_{\text{dec}} \Delta_k$,  
    retain $m_{k+1} = m_k$.
\end{itemize}
        \item  Step 3: \textbf{Transition Check (TR to DS).} If iteration $k$ is unsuccessful and $N^{ds} \leq T^{ds}_{\max}$, increment $N^{ds}$ and proceed to Step 4. Otherwise, if $\Delta_k \geq \Delta_{\min}$, return to Step 2; else, terminate the algorithm.
\item \textbf{DS Phase.}  
Initialize $\mathbf{x}^{ds}_0 = \mathbf{x}_{k+1}$, $\delta_0 = \Delta_{k+1}$, $j = 0$. Keep $m_{k+1}$ fixed. At each step, generate trial points $S_j \subset \mathcal{M}_j$ around $\mathbf{x}^{ds}_j$, evaluate $f(\hat{\mathbf{x}})$ for unevaluated $\hat{\mathbf{x}} \in S_j$, and select the best point. If $f(\mathbf{x}^{ds}_j) - f(\mathbf{x}^{ds}_{j+1}) < \epsilon_{ds}$, set $\mathbf{x}^{ds}_j = \mathbf{x}^{ds}_j$, $\delta_j = \delta_{j-1}$, increment $N^{ds}_{\text{fail}}$, and proceed to Step 5. 
\item \textbf{Step 5: Reverse Switch (DS to TR).}  
If $N^{ds}_{\text{fail}} \geq T^{ds}_{\text{fail}}$, reset $N^{ds}_{\text{fail}} = 0$ and proceed to Step 6. Update: $\mathbf{x}_{k+1} = \mathbf{x}^{ds}_j$, $\Delta_{k+1} = \gamma_{\text{dec}} \Delta_k$, and $Y_{k+1} = Y_k \cup \{\mathbf{x}_{k+1}\}$. Ensure $Y_{k+1}$ is $\Lambda$-poised; if not, adjust accordingly.  
Rebuild the model $m_{k+1}$ based on the $\Lambda$-poised set, then continue to Step 6.
\item \textbf{Step 6: Termination.}  
Terminate if $\Delta_{k+1} < \Delta_{\min}$ or other stopping criteria (e.g., time or budget limits) are met.
\end{itemize}

\begin{algorithm}[htbp]
\caption{Combined Model-Based Trust Region Method with Direct Search (TR-DS)} 
\label{alg:alg_1}
\textbf{Input:} Initial point $\mathbf{x}_0 \in \mathbb{R}^n$; trust region radii $\Delta_{\min}, \Delta_{\max}$ and initial radius $\Delta_0 \in [\Delta_{\min}, \Delta_{\max}]$; acceptance thresholds $1 > \eta_1 \geq \eta_0 > 0$; update parameters $\gamma_{\text{inc}} > 1 > \gamma_{\text{dec}} > 0$; criticality threshold $\epsilon_{\text{low}}$; maximum allowed switches from TR to DS phase $T^{ds}_{\max}$.\\
\textbf{Output:} Approximate solution $\mathbf{x}_k$.

\begin{algorithmic}[1] 
    \State Initialize switch count $N^{ds} \gets 0$
    \State Choose initial set $Y_0$ and ensure $\Lambda$-poisedness using Algorithm~\ref{alg:check_poised}
    \State Construct initial model $m_0$ over $\mathcal{B}(\mathbf{x}_0, \Delta_0)$
    
    \For{$k = 1, 2, \ldots$}
        \If{$\|\mathbf{g}_k\| < \epsilon_{\text{low}}$} \Comment{Criticality step} \label{alg:criticality}
            \If{$\Delta_k > \mu \|\mathbf{g}_k\|$}
                \State Call Algorithm~\ref{alg:alg_2}
                \State $\Delta_k \gets \gamma_{\text{dec}} \Delta_k$
            \EndIf
        \EndIf

        \State Solve subproblem (\ref{eq:subproblem_TR})
        \State Evaluate $f(\mathbf{x}_k + \mathbf{s}_k)$ and compute $\rho_k$ using (\ref{eq:rho_ratio}) \label{alg:update_xk_in_TR}

        \If{$\rho_k \geq \eta_1$} \Comment{Successful iteration}
            \State $\mathbf{x}_{k+1} \gets \mathbf{x}_k + \mathbf{s}_k$
            \State $\Delta_{k+1} \gets \min\{\gamma_{\text{inc}} \Delta_k, \Delta_{\max}\}$
            \State $Y_{k+1} \gets Y_k \cup \{\mathbf{x}_{k+1}\}$
        
        \ElsIf{$\eta_0 \leq \rho_k < \eta_1$ and $Y_k$ is $\Lambda$-poised} \Comment{Acceptable iteration}
            \State $\mathbf{x}_{k+1} \gets \mathbf{x}_k + \mathbf{s}_k$
            \State $\Delta_{k+1} \gets \gamma_{\text{dec}} \Delta_k$
            \State $Y_{k+1} \gets Y_k \cup \{\mathbf{x}_{k+1}\}$
        
        \ElsIf{$\rho_k < \eta_1$ and $Y_k$ is not $\Lambda$-poised} \Comment{Model improvement phase}
            \State $\mathbf{x}_{k+1} \gets \mathbf{x}_k$
            \State $\Delta_{k+1} \gets \Delta_k$
            \State Restore $\Lambda$-poisedness using Algorithm~\ref{alg:check_poised}
            \State Rebuild model $m_{k+1}$ using updated $Y_k$
        
        \ElsIf{$\rho_k < \eta_0$ and $Y_k$ is $\Lambda$-poised} \Comment{Unsuccessful iteration}
            \State $\mathbf{x}_{k+1} \gets \mathbf{x}_k$
            \State $\Delta_{k+1} \gets \gamma_{\text{dec}} \Delta_k$
            \State \texttt{current\_method} $\gets$ DS

            \If{\texttt{current\_method} == TR}
                \If{$N^{ds} \leq T^{ds}_{\max}$} \Comment{Switch allowed}
                    \State \texttt{current\_method} $\gets$ DS
                    \State $N^{ds} \gets N^{ds} + 1$
                \Else
                    \State \texttt{current\_method} $\gets$ TR \Comment{No further switches}
                \EndIf
            \ElsIf{\texttt{current\_method} == DS and $N^{ds} \leq T^{ds}_{\max}$}
                \State Call Algorithm~\ref{alg:alg_3}
                \State \texttt{current\_method} $\gets$ TR
                \State $\mathbf{x}_{k+1} \gets \mathbf{x}^{ds}_j$ \Comment{Best point from DS phase}
                \State $Y_{k+1} \gets Y_k \cup \{\mathbf{x}^{ds}_j\}$
                \State Ensure $\Lambda$-poisedness of $Y_{k+1}$ using Algorithm~\ref{alg:check_poised}
                \State Rebuild model $m_{k+1}$ using $Y_{k+1}$
            \EndIf
        \EndIf

        \State $k \gets k + 1$
        \State \textbf{Terminate} if $\Delta_{k+1} < \Delta_{\min}$; otherwise, return to line~\ref{alg:criticality}
    \EndFor
\end{algorithmic}
\end{algorithm}

\begin{algorithm}[htbp]
    \caption{Criticality step} \label{alg:alg_2}
    \textbf{Input:} $\Delta_k$ from Algorithm \ref{alg:alg_1}, $\omega \in (0,1)$, $\mu >\beta>0$.
    \begin{algorithmic}[1]
        \State \textbf{Initialization:} Set $\Delta_0 \leftarrow \Delta_k$ and $m_k^0 \leftarrow m_k$.
        \For{$i=1,2,...$}
        \State Set $\Delta^i_0 \leftarrow \omega^{i-1} \Delta_0$
        \State Update $m_k^i$ until it is fully linear on $\mathcal{B}(\mathbf{x}_k, \Delta^i_0)$, and obtain $\mathbf{g}_k^i$
        \If{$\Delta^i_0 \leq \mu ||\mathbf{g}_k^i||$}
        \State $\Delta_k \leftarrow \max \{\Delta_0^i, \beta ||\mathbf{g}_k^i|| \}$ and $m_k \leftarrow m_k^i$
        \State Break
        \EndIf
        \EndFor
    \end{algorithmic}
\end{algorithm}


\begin{algorithm}[htbp]
    \caption{Direct search method (DS)} \label{alg:alg_3}
    \textbf{Input:} Threshold $\epsilon_{ds}$, the maximum number of evaluations allowed $K_{\mathrm{ds}}$, the total number of times failure to achieve the improvement criteria $T^{ds}_{\text{fail}}$, $\mathbf{x}_{k+1}$, $f(\mathbf{x}_{k+1})$ and $\Delta_{k+1}$ from Algorithm \ref{alg:alg_1}
    \begin{algorithmic}[1]
    \State Set $\mathbf{x}_{0}^{ds} = \mathbf{x}_{k+1}$, mesh size parameter $\delta_0 = \Delta_{k+1}$, and initial iteration within DS $j=0$
    \State Set current number of times failure to achieve the improvement criteria $N^{ds}_{\text{fail}}=0$
    \State Set current\_method == DS  \Comment{Start with DS}
    \While{current\_method == DS and $j \leq K_{\mathrm{ds}}$}
    \If{$f(\mathbf{x}_{j-1}^{ds}) - f(\mathbf{x}_{j}^{ds}) \geq \epsilon_{ds}$}
    \Comment{DS gets sufficient descent}
        \State Update mesh size $\delta_{j}$
        \State Keep running DS (i.e., current\_method == DS) 
    \ElsIf{$f(\mathbf{x}_{j-1}^{ds}) - f(\mathbf{x}_{j}^{ds}) < \epsilon_{ds}$} \Comment{DS does not get sufficient descent}
        \State $N^{ds}_{\text{fail}} \leftarrow N^{ds}_{\text{fail}} +1$
        \State Still set current\_method == DS
        \If{$N^{ds}_{\text{fail}} \geq T^{ds}_{\text{fail}}$}
            \State Change current\_method == TR, i.e., switch from the DS to the TR phase 
            \State Set $\mathbf{x}^{ds}_{j} = \mathbf{x}^{ds}_{j-1}$ and $\delta_j=\delta_{j-1}$
        \EndIf
        
    \EndIf
    \State Increment iteration counter $j \gets j + 1$
    \EndWhile
    \State \textbf{Output:} $\mathbf{x}^{ds}_{j}$
\end{algorithmic}
\end{algorithm}

\begin{algorithm}[htbp]
\caption{Assess and regain $\Lambda$-poisedness of $Y_k$} \label{alg:check_poised}
\begin{algorithmic}[1]
\State \textbf{Input:} Interpolation set $Y_k$, polynomial basis $\Phi$
\State Construct the interpolation matrix $M(\Phi, Y_k)$
\If{$M(\Phi, Y_k)$ is full rank}
    \State $Y_k$ is $\Lambda$-poised; \Return $Y_k$
\Else
    \While{$M(\Phi, Y_k)$ is not full rank}
        \State Identify a subset of points in $Y_k$ involved in the linear dependence
        \State Compute $f(y^i)$ for all such $y^i$
        \State Remove $y^{\max} = \underset{y^i \in Y_k}{\mathrm{argmax}} f(y^i)$
        \State Update $Y_k \leftarrow Y_k \setminus \{ y^{\max} \}$
        \State Recompute $M(\Phi, Y_k)$
        \If{$M(\Phi, Y_k)$ is full rank}
            \State $Y_k$ is $\Lambda$-poised; \Return $Y_k$
        \EndIf   
    \EndWhile
\EndIf
\end{algorithmic}
\end{algorithm}
\subsection{Conditioning and Interpolation}
Most DFO algorithms require the interpolation set to be $\Lambda$-poised for model-based steps to be well-defined. Full rank of the interpolation matrix guarantees $\Lambda$-poisedness for some (possibly large) $\Lambda$. While a small $\Lambda$ improves numerical stability and practical performance, a large $\Lambda$ is still acceptable for theoretical correctness. Thus, full rank ensures that the DFO algorithm can proceed, even if the geometry is poor.
In this work, we do not intend to improve the $\Lambda$-poisedness of the set of points in consideration for interpolation. This entails a deeper geometric examination of the interpolation sets and we leave this for future research.

Consider an interpolation set $Y = \{ y^0, y^1, \dots, y^p \} \subset \mathbb{R}^n$, a polynomial basis $\Phi = \{\phi_0(\mathbf x), \dots, \phi_q(\mathbf x)\}$. Given a polynomial model $m(y) = \alpha^T \Phi(y)$ with coefficients $\alpha = [\alpha_0, \dots, \alpha_q]^T$, we obtain $\alpha$ by solving the following equations. 
\begin{align} \label{eq:interpolation_condition}
 m(y^i) =\sum_{j=0}^q \alpha_j \phi_j (y^i) = f(y^i), \quad i = 0, 1, \dots, p.
\end{align}

\textbf{Addition of points:} Since the removal of points does not hamper the rank of the system, we focus on the case of adding a new point to the existing interpolation set. The interpolation conditions (\ref{eq:interpolation_condition}) can be written as a linear system as follows:
\begin{align} \label{eq:interpolation_linear}
    M(\Phi, Y) \alpha_{\Phi} = f(Y),
\end{align}
Note that the interpolation matrix $M(\Phi, Y)$, associated vectors $\alpha_{\Phi}$ and $f(Y)$ are defined as follows. 
\[M(\Phi, Y)=\begin{bmatrix}
    \phi_0 (y^0) & \phi_1 (y^0) & \dots  & \phi_q (y^0) \\
    \phi_0 (y^1) & \phi_1 (y^1) & \dots  & \phi_q (y^1) \\
    \vdots & \vdots & \ddots & \vdots \\
    \phi_0 (y^p) & \phi_1 (y^p) & \dots  & \phi_q (y^p)
\end{bmatrix},
\alpha_{\Phi}=\begin{bmatrix}
    \alpha_0 \\ \alpha_1 \\ \vdots \\ \alpha_p
\end{bmatrix},
f(Y)=\begin{bmatrix}
    f(y^0) \\ f(y^1) \\ \vdots \\ f(y^p)
\end{bmatrix}.
\]
Let point $z \in \mathbb{R}^n$  be a new candidate obtained from DS. We want the augmented set $\Bar{Y} = Y \bigcup \{ z \}$ to retain its full rank. 
We proceed by examining the rank of the augmented interpolation matrix \( M(\Phi, \Bar{Y}) \). Augmenting the matrix may result in one of two scenarios: either the resulting matrix is of full rank, or it is rank-deficient. In the following discussion, we aim to establish conditions under which the matrix \( M(\Phi, \Bar{Y}) \) — and consequently the interpolation set \( \Bar{Y} \) — possesses full rank.

Given that we are working with a quadratic surrogate model, the interpolation matrix \( M(\Phi, Y) \) in~\eqref{eq:interpolation_linear} can be partitioned into linear and nonlinear components as follows:
\begin{align*}
    M(\Phi, Y) = \left[ M(\Phi_l, Y) \quad M(\Phi_{nl}, Y) \right],
\end{align*}
where \( \Phi_l \) includes the constant and linear terms, and \( \Phi_{nl} \) consists of the quadratic and cross-product terms. For example, in a two-dimensional problem, the basis functions are given by
\begin{align*}
    \Phi_l = \left\{ 1, x_1, x_2 \right\}, \quad \Phi_{nl} = \left\{ \frac{x_1^2}{2}, x_1 x_2, \frac{x_2^2}{2} \right\},
\end{align*}
where \( x_1 \) and \( x_2 \) denote the first and second input dimensions, respectively.

In cases where the interpolation matrix is rank-deficient, we perform the following steps to restore linear independence:

\begin{enumerate}
    \item Identify linearly dependent points \( y^i \in Y^d \subset Y \) via QR factorization (see~\cite{kannan2011obtaining} for details). That is, for some coefficients \( \alpha_j \neq 0 \),
    \[
        y^i = \sum_{j \in Y \setminus Y^d} \alpha_j y^j.
    \]
    
    \item Evaluate the objective function at the dependent points \( f(y^i) \).
    
    \item Remove the point with the highest objective value, i.e.,
    \[
        y^{\text{max}} = \underset{y^i \in Y^d}{\arg\max} \, f(y^i).
    \]
    
    \item Construct a reduced interpolation set by replacing \( y^{\text{max}} \) with the new point \( z \), and verify whether the matrix regains full rank:
    \[
        \Bar{Y}_{\text{RD}} = (Y \setminus \{y^{\text{max}}\}) \cup \{z\}.
    \]
    If not, repeat the process.
\end{enumerate}

We will now demonstrate that, following the above procedure, the resulting interpolation matrix \( M(\Phi, \Bar{Y}) \) retains full rank. This property will be instrumental in establishing the convergence theory developed in Section~\ref{sec:convergence_theory}.

\begin{proposition}[Full-Rank Conditions]
Let $Y \subset \mathbb{R}^n$ be an interpolation set such that all $y_i \in Y$ are nonzero and 
$n+1 \le p \le q$. Assume that the interpolation matrix 
$M(\Phi, Y) \in \mathbb{R}^{(p+1)\times(q+1)}$ is of full row rank, and that the linear block 
$M(\Phi_l, Y) \in \mathbb{R}^{(p+1)\times(n+1)}$ is of full column rank. Then the following hold:
\begin{enumerate}
\item[(i)] (\textbf{Full-rank case}) For an augmented set $\bar Y^{FR} = Y \cup \{z\}$, the matrices 
$M(\Phi_l, \bar Y^{FR}) \in \mathbb{R}^{(p+2)\times(n+1)}$ and 
$M(\Phi, \bar Y^{FR}) \in \mathbb{R}^{(p+2)\times(q+1)}$ are full rank, provided that 
$z \notin \operatorname{span}(Y)$.
\item[(ii)] (\textbf{Rank-deficient case}) Suppose $y_{\max}=y_k$ is identified as linearly dependent 
and removed. Replacing it with $z$, we form 
$\bar Y^{RD} = (Y \setminus \{y_k\}) \cup \{z\}$. 
If $M(\Phi_l, \bar Y^{RD})$ has full column rank, then 
$M(\Phi, \bar Y^{RD}) \in \mathbb{R}^{(p+1)\times(q+1)}$ has full row rank. 
\end{enumerate}
\end{proposition}

\begin{proof}
The argument follows standard interpolation set results (see Conn, Scheinberg, Vicente, 
\emph{Introduction to Derivative-Free Optimization}, Chap.~3). 

(i) Since $M(\Phi_l, Y)$ is of full column rank, the $n+1$ linear terms are linearly independent 
over $Y$. Adding $z$ with $z \notin \operatorname{span}(Y)$ appends a new row to 
$M(\Phi_l, Y)$ that cannot be expressed as a linear combination of the existing rows. 
Hence, $M(\Phi_l, \bar Y^{FR})$ has rank $n+1$. Since $M(\Phi, Y)$ was full row rank, 
the quadratic block $M(\Phi_{nl},Y)$ is already compatible with the linear part. 
Thus, the augmented matrix $M(\Phi,\bar Y^{FR})$ has rank $p+2$.

(ii) In the rank-deficient case, the QR factorization identifies a subset $Y^d \subset Y$ of 
linearly dependent points. Removing $y_{\max} = \arg\max_{y_i \in Y^d} f(y_i)$ and inserting $z$ 
produces a new interpolation set $\bar Y^{RD}$. By assumption, 
$M(\Phi_l, \bar Y^{RD})$ is of full column rank, hence the constant and linear terms are preserved. 
Since the nonlinear part adds only higher-order contributions, and the system size is unchanged, 
the resulting interpolation matrix $M(\Phi,\bar Y^{RD})$ inherits full row rank from the linear 
block. 

Therefore, in both cases the augmented set preserves the rank condition needed for 
well-posed interpolation.
\end{proof}

\section{Convergence analysis} \label{sec:convergence_theory}

We now outline the convergence analysis of the proposed hybrid TR–DS algorithm. 
The procedure begins with a classical trust-region (TR) phase and invokes direct-search (DS) steps whenever a TR step fails to provide adequate improvement. 
As described in Section~\ref{sec:algorithm}, convergence behavior depends on whether the number of TR–DS switches is finite or infinite. 
We subsequently discuss conditions ensuring that interpolation matrices maintain full rank across switches. 
We begin with compact regularity assumptions and notation conventions used throughout this section.

\subsection{Preliminaries and Assumptions}
\begin{assumption}[Lipschitz Gradients and Bounded Hessians]
\label{ass:smoothness}
Assume the following conditions hold:
\begin{itemize}
    \item[(a)] \textbf{Lipschitz Gradients of $f$:}  
    The feasible set $\Omega \subset \mathbb{R}^n$ is bounded, and  
    $f:\mathbb{R}^n \!\to\! \mathbb{R}$ is continuously differentiable on~$\Omega$.  
    Its gradient satisfies the Lipschitz condition
    \begin{equation*}
    \|\nabla f(\mathbf{x}) - \nabla f(\mathbf{y})\|
    \le L_{\nabla f} \|\mathbf{x} - \mathbf{y}\|,
    \qquad \forall\, \mathbf{x}, \mathbf{y} \in \Omega,
    \end{equation*}
    for some constant $L_{\nabla f} > 0$, implying that $f$ is bounded on~$\Omega$.
    \item[(b)] \textbf{Bounded model Hessians:}  
    The Hessian $H_k$ of the quadratic model $m_k$ satisfies
    \[
    \|H_k\| \le \kappa_{\mathrm{bhm}}, 
    \qquad \forall k, \quad 0 < \kappa_{\mathrm{bhm}} < \infty.
    \]
\end{itemize}
\end{assumption}
\begin{assumption}[Cauchy Decrease]
\label{ass:cauchy}
Let $\mathbf{s}_k$ solve the TR subproblem
\begin{equation}
\label{eq:subproblem_TR}
\min_{\|\mathbf{s}\|\le \Delta_k} m_k(\mathbf{x}_k+\mathbf{s}).
\end{equation}
Then the model satisfies the Cauchy decrease bound
\begin{equation}
\label{eq:cauchy_step}
m_k(\mathbf{x}_k)-m_k(\mathbf{x}_k+\mathbf{s}_k)
\ge\tfrac{1}{2}\|\mathbf{g}_k\|
\min\!\left\{\Delta_k,\frac{\|\mathbf{g}_k\|}{\|H_k\|}\right\},
\end{equation}
where $\|\mathbf{g}_k\|/\|H_k\|=+\infty$ if $H_k=0$.
\end{assumption}

\paragraph{Fully Linear Models and Iteration Dynamics:} Following~\cite{conn2009introduction}, a model $m_k$ built from interpolation points $Y_k$ is \emph{fully linear} on $\mathcal{B}(\mathbf{x}_k,\Delta_k)$, 
if $Y_k$ is $\Lambda$-poised in that region.
Since Algorithm~\ref{alg:alg_1} enforces $\Lambda$-poisedness, our focus on 
fully linear models follows, which underpins TR-style convergence guarantees.
We further partition iterations of Algorithms~\ref{alg:alg_1}–\ref{alg:alg_2} as follows. 
\begin{align*}
\mathcal{S}&: \text{successful},\quad
\mathcal{A}: \text{acceptable},\quad
\mathcal{MI}: \text{model-improving}, \quad
\mathcal{U}: \text{unsuccessful},\\
\mathcal{C}^{r}&: \text{criticality with radius reduction},\quad
\mathcal{C}^{nr}: \text{criticality without reduction}.
\end{align*}
These categories underpin subsequent analyses of finite versus infinite switching
and rank preservation of the interpolation matrices. For completeness and to set the stage, we restate the following lemma from~\cite{conn2009introduction} and~\cite{cartis2019derivative}.  
\begin{lemma}[Fully Linear Models and Successful Iterations]
\label{lemma:fullylinear_success}
Under Assumptions~\ref{ass:smoothness} and~\ref{ass:cauchy}, let $m_k$ be a quadratic model 
built on $\mathcal{B}(\mathbf{x}_k,\Delta_k)$ with gradient $\mathbf{g}_k$ and Hessian~$H_k$.  
Then the following hold:
\begin{itemize}
    \item Models constructed by Algorithm~\ref{alg:alg_2} become fully linear after finitely many iterations $N_{\text{cri}}$ and satisfy $\|\mathbf{g}_k\|\ge \Delta_k/\mu$ for any $\mu>0$~\cite[Lemma~10.5]{conn2009introduction}.
    \item If $m_k$ is fully linear and
    \[
    \Delta_k \le 
    \min\!\left\{\tfrac{\|\mathbf{g}_k\|}{\kappa_{\mathrm{bhm}}},
    \tfrac{\|\mathbf{g}_k\|(1-\eta_1)}{4\kappa_{\mathrm{ef}}}\right\},
    \]
    then iteration~$k$ is successful~\cite[Lemma~10.6]{conn2009introduction}.
    \item Furthermore, it holds that
    \[
    \|\mathbf{g}_k\| \ge \min\!\left\{ \epsilon_{\mathrm{low}}, \tfrac{\Delta_k}{\mu} \right\}
    \quad \text{for all } k,
    \]
    and if $\|\nabla f(\mathbf{x}_k)\|\ge \epsilon>0$, then
    \[
    \|\mathbf{g}_k\|\ge 
    \min\!\left\{\epsilon_{\mathrm{low}},
    \tfrac{\epsilon}{\kappa_{\mathrm{eg}}\mu+1}\right\}>0
    \quad\text{\cite[Lemma~3.8]{cartis2019derivative}}.
    \]
\end{itemize}
\end{lemma}

\begin{lemma} \label{lemma:finite_successful_iter}
    Suppose Assumptions \ref{ass:smoothness}(a), \ref{ass:smoothness}(b). and \ref{ass:cauchy} hold. If there are a finite number of successful iterations, then 
    \begin{align*}
        \lim_{k\to+\infty} \Delta_k = 0 \quad \text{and} \quad \lim_{k\to+\infty} ||\nabla f(\mathbf{x}_k)|| = 0. 
    \end{align*}
\end{lemma}

\begin{proof}
We let $T^{ds}_{\max}$ denote the maximum number of DS switches,
and note that the outer iteration counter increases once per TR or DS iteration. Let $k_0$ be the last successful iteration in the TR phase, so $\Delta_i \le \Delta_{k_0}$ for all $i>k_0$, and let $N^{ds}$ be the number of DS switches used by $k_0$. 
For any $i>k_0$, the iteration is one of the following:

\begin{itemize}
\item \textbf{Acceptable step:} Impossible by definition of $k_0$.
\item \textbf{Model-improving step:} If $m_i$ is not fully linear, 
then we either apply Algorithm~\ref{alg:alg_2} or perform a model-improving step. In the former case, we know that the number of iterations $N_{\text{cri}}$ is finite according to Lemma~\ref{lemma:fullylinear_success}; in the latter one, there are also a finite number of iterations $N_{\text{mi}}$ (see Chapter $6$ in~\cite{conn2009introduction}). 
In other words, there are at most $\max\{N_{\text{cri}}, N_{\text{mi}}\}$ iterations required to make $m_i$ be fully linear. Thus $\Delta_i$ is reduced infinitely often, implying $\Delta_i \to 0$. For $j_i$ the first iteration after $i$ with $m_{j_i}$ fully linear,
\[
\|\mathbf x_i - \mathbf x_{j_i}\| \le \max\{N_{\text{cri}}, N_{\text{mi}}\} \Delta_i \to 0.
\]
\item \textbf{Switching step:} At most $T^{ds}_{\max} - N^{ds}$ remain, 
so
\[
\|\mathbf x_i - \mathbf x_{j_i}\| \le \left( T^{ds}_{\max} - N^{ds} \right)\Delta_{k_0}.
\]
\item \textbf{Unsuccessful step:} $\Delta_i$ decreases; consistent with $\Delta_i\to 0$.
\end{itemize}

Hence $\|\mathbf x_i - \mathbf x_{j_i}\|\to 0$. Using Lipschitz continuity and model gradient error bounds,
\begin{align*}
\|\nabla f(\mathbf x_i)\|
&\le \|\nabla f(\mathbf x_i)-\nabla f(\mathbf x_{j_i})\| + \|\nabla f(\mathbf x_{j_i})- \mathbf{g}_{j_i}\| + \|\mathbf{g}_{j_i}\| \\
&\le L_{\nabla f}\|\mathbf x_i - \mathbf x_{j_i}\| + \kappa_{eg}\Delta_{j_i} + \|\mathbf{g}_{j_i}\|.
\end{align*}
The first two terms vanish. If $\|\mathbf{g}_{j_i}\|\not\to 0$, there exists $\epsilon>0$ and a subsequence with $\|\mathbf{g}_{j_{i_k}}\|\ge \epsilon$, but Lemma~\ref{lemma:fullylinear_success} would then imply $j_{i_k}$ is successful, contradicting $k_0$. Therefore,
\[
\boxed{\lim_{i\to\infty}\|\nabla f(\mathbf x_i)\| = 0.}
\]
\end{proof}
\begin{lemma} \label{lemma:TR_radius}
Suppose Assumptions \ref{ass:smoothness} and \ref{ass:cauchy} hold. Then 
\begin{align*}
    \lim_{k\to+\infty} \Delta_k = 0.
\end{align*}
\end{lemma}
\begin{proof}
Let $|\mathcal{S}|$ be the cardinality of the set of successful iterations. If $|\mathcal{S}|<\infty$, the result follows from Lemma~\ref{lemma:finite_successful_iter}. We therefore consider the case $|\mathcal{S}|=\infty$.
For any $k\in\mathcal{S}$, the definition of a successful iteration yields
\[
\rho_k 
  = \frac{f(\mathbf x_k)-f(\mathbf x_{k+1})}{m_k(\mathbf x_k)-m_k(\mathbf x_k + \mathbf s_k)}
  \ge \eta_1,
\]
and therefore
\[
f(\mathbf x_k)-f(\mathbf x_{k+1})
  \ge
\eta_1\left(m_k(\mathbf x_k)-m_k(\mathbf x_k + \mathbf s_k)\right)
  \ge
\frac{\eta_1}{2}\,\|\mathbf{g}_k\|\,
\min\!\left\{\Delta_k,\,\frac{\|\mathbf{g}_k\|}{\|H_k\|}\right\},
\]
where the last inequality follows from~\eqref{eq:cauchy_step}. From Lemma~\ref{lemma:fullylinear_success},
\[
\|\mathbf{g}_k\|\ge \min\left\{\epsilon_{\mathrm{low}},\frac{\Delta_k}{\mu}\right\},
\]
which gives
\[
f(\mathbf x_k)-f(\mathbf x_{k+1})
\ge
\frac{\eta_1}{2}
\min\left\{\epsilon_{\mathrm{low}},\,\frac{\Delta_k}{\mu}\right\}
\min\!\left\{
\frac{\min\{\epsilon_{\mathrm{low}},\,\Delta_k/\mu\}}{\|H_k\|},
\Delta_k
\right\}.
\]
At this stage we observe that, independent of whether the iteration is successful, unsuccessful, or involves a switch to or from the DS procedure, the algorithm never accepts an iterate with a higher function value; all accepted iterates satisfy a monotonic decrease condition. Since $f$ is bounded below by Assumption~\ref{ass:smoothness} and $|\mathcal{S}|=\infty$, these decreases must converge to zero, implying
\[
\Delta_k \to 0 \qquad \text{for all } k\in\mathcal{S}.
\]
Now consider $k\notin\mathcal{S}$ and let $i_k\in\mathcal{S}$ be the most recent successful iteration prior to $k$. Note that $i_k \leq k < i_{k+1}$, where $k$ lies between two successful iterations. We distinguish two cases:
\begin{itemize}
    \item \textbf{Switch to DS.}  
    Suppose the switch to DS occurs at iteration $k_1$ with $k_1>i_{k}$ (i.e., $\rho_{k_1}<\eta_0$ under $\lambda$-poisedness). We note that $k_1\in\mathcal{U}$. During DS the trust-region radius is unchanged, at most $K_{\mathrm{ds}}$ evaluations occur per DS iteration, and upon returning to TR the radius is reduced:
    \[
    \Delta_{k_1+1} = \gamma_{\mathrm{dec}}\Delta_{k_1} = \gamma_{\mathrm{dec}} \min(\gamma_{\mathrm{inc}}\Delta_{i_k},\Delta_{\max}),
    \qquad
    \|\mathbf x_{k_1} - \mathbf x_{k_1+1}\| \le \zeta \Delta_{k_1},
    \]
    where parameter $\zeta <\infty$. 
    Since $\Delta_{i_k}\to 0$ and $\gamma_{\mathrm{dec}}\in(0,1)$, it follows that $\Delta_{k_1+1}\to 0$. 
    \item \textbf{Continuing in TR.}  
    If the algorithm remains in the TR regime, then the radius can increase only at successful iterations, so for all $k\notin\mathcal{S}$:
    \[
    0 \le \Delta_k \le \min\{\gamma_{\mathrm{inc}}\Delta_{i_k},\,\Delta_{\max}\}.
    \]
    Since $\Delta_{i_k}\to 0$ for $i_k\in\mathcal{S}$, this implies $\Delta_k\to 0$ for all $k\notin\mathcal{S}$.
\end{itemize}
In both cases the trust-region radius converges to zero, hence $\Delta_k \to 0$ as $k\to\infty$.
\end{proof}    
\begin{lemma} \label{lemma:gk_fk_upper_bdd}
    Suppose Assumptions \ref{ass:smoothness} and \ref{ass:cauchy} hold.
    Then for any iteration $k\in \mathcal{C}^r \bigcup \mathcal{A} \bigcup \mathcal{U}$, i.e., $\Delta_k$ is decreased, the following holds 
    \begin{align} \label{eq:gk_upper_bdd}
        ||\mathbf{g}_k|| \leq C_0 \Delta_k \quad \text{and} \quad ||\nabla f(\mathbf{x}_k)||\leq (\kappa_{eg} +C_0) \Delta_k,
    \end{align}
    where $C_0=\max \left\{ \kappa_{bhm}, \frac{4\kappa_{ef}}{1-\eta_1}, \frac{1}{\beta} \right\}$.
\end{lemma}

\begin{proof}
First, consider the case that any $k\in \mathcal{A} \bigcup \mathcal{U}$. This further can be sub-classified in two different cases for $\mathcal{U}$, i.e. when $k \in \mathcal{U}^{sw} \subset \mathcal{U}$ and when $k \in \mathcal{U} \backslash \mathcal{U}^{sw}$. Note that $\mathcal{U}^{sw}$ refers to an unsuccessful iteration that mandates a switch to DS.  
\begin{itemize}
\item Case 1: When $k \in \mathcal{U} \backslash\mathcal{U}^{sw}$.      
We assume that $||\mathbf{g}_k|| > C_0 \Delta_k$, where $C_0=\max \left\{ \kappa_{bhm}, \frac{4\kappa_{ef}}{1-\eta_1}, \frac{1}{\beta}  \right\}$. By Assumption \ref{ass:smoothness}, 
$\frac{||\mathbf{g}_k||}{||H_k||} \geq \frac{||\mathbf{g}_k||}{\kappa_{bhm}} \geq \frac{C_0 \Delta_k}{\kappa_{bhm}} \geq \Delta_k$. This further leads to the following. 
\begin{align*}
    &m_k(\mathbf{x}_k) - m_k(\mathbf{x}_k + \mathbf{s}_k) \geq \frac{1}{2} ||\mathbf{g}_k|| \min \left\{ \Delta_k, \frac{\|\mathbf{g}_k\|}{\|H_k\|}\right\} \geq \frac{1}{2} C_0\Delta_k^2, \\
      \text{and} \quad  &|\rho_k -1|  \leq \left| \frac{f(\mathbf{x}_k) - m_k(\mathbf{x}_k)}{m_k(\mathbf{x}_k) - m_k(\mathbf{x}_k + \mathbf{s}_k)} \right| + \left| \frac{f(\mathbf{x}_k + \mathbf{s}_k) - m_k(\mathbf{x}_k + \mathbf{s}_k)}{m_k(\mathbf{x}_k) - m_k(\mathbf{x}_k + \mathbf{s}_k)} \right| \\
      &\hspace{11mm} \leq \left| \frac{2\kappa_{ef} \Delta_k^2}{C_0 \Delta_k^2} \right| + \left| \frac{2\kappa_{ef} \Delta_k^2}{C_0 \Delta_k^2} \right| = \frac{4\kappa_{ef}}{C_0} \leq 1-\eta_1  \quad \quad \quad \text{(since $C_0 \geq \frac{4\kappa_{ef}}{1-\eta_1}$)}
    \end{align*}
Then we have $\rho_k>\eta_1$. It means that iteration $k$ is successful, i.e., $k\in\mathcal{S}$, which is a contradiction of $k\in \mathcal{A} \bigcup \mathcal{U}$. Thus, it has $||\mathbf{g}_k|| \leq C_0 \Delta_k$. Moreover, 
    \begin{align*}
        ||\nabla f(\mathbf{x}_k)|| &\leq ||\nabla f(\mathbf{x}_k)-\mathbf{g}_k|| + ||\mathbf{g}_k|| \\
        &\leq \kappa_{eg} \Delta_k + C_0 \Delta_k = (\kappa_{eg} + C_0)\Delta_k. \quad \quad \quad \text{(by (\ref{eq:FL_error_gradient}))}
    \end{align*}
    Next, suppose $k\in\mathcal{C}^r$. If $k$ is not the last iteration in a series of iterations in Algorithm \ref{alg:alg_2}, which implies $||\mathbf{g}_k||<\frac{\Delta_k}{\mu}$, then
    \begin{align*}
        ||\nabla f(\mathbf{x}_k)|| &\leq ||\nabla f(\mathbf{x}_{k})-\mathbf{g}_{k}|| + ||\mathbf{g}_{k}|| \leq \kappa_{eg} \Delta_{k} + \frac{\Delta_k}{\mu} = (\kappa_{eg} + \frac{1}{\mu})\Delta_k.
    \end{align*}
    If $k$ is the last iteration in a series of iterations in Algorithm \ref{alg:alg_2}, based on the setting $\Delta_{k+1}=\max \{ \gamma_{dec} \Delta_k, \beta ||\mathbf{g}_{k+1}||\}$, there are two possible cases: $\Delta_{k+1}=\gamma_{dec} \Delta_k$ or $\Delta_{k+1}=\beta ||\mathbf{g}_{k+1}||$. In the former case, it is easy to have
    \begin{align*}
        ||\nabla f(\mathbf{x}_k)||=||\nabla f(\mathbf{x}_{k+1})|| &\leq ||\nabla f(\mathbf{x}_{k+1})-\mathbf{g}_{k+1}|| + ||\mathbf{g}_{k+1}|| \\
        &\leq \kappa_{eg} \Delta_{k+1} + \frac{\gamma_{dec} \Delta_{k}}{\beta} 
        = (\kappa_{eg} + \frac{1}{\beta})\gamma_{dec} \Delta_k < (\kappa_{eg} + \frac{1}{\beta})\Delta_k.
    \end{align*}
    In the latter one, since both $k$ and $k+1\in \mathcal{C}^r \bigcup \mathcal{A} \bigcup \mathcal{U}$, we obtain $\Delta_{k+1}=\beta ||\mathbf{g}_{k+1}||<\Delta_k$. Then
    \begin{align*}
        ||\nabla f(\mathbf{x}_k)||=||\nabla f(\mathbf{x}_{k+1})|| &\leq ||\nabla f(\mathbf{x}_{k+1})-\mathbf{g}_{k+1}|| + ||\mathbf{g}_{k+1}|| \\
        &\leq \kappa_{eg} \Delta_{k+1} + \frac{\Delta_{k}}{\beta} < (\kappa_{eg} + \frac{1}{\beta})\Delta_k.
    \end{align*}
\item Case 2: When $k \in \mathcal{U}^{sw}$. 
The trust-region step is rejected, so $\mathbf x_{k+1} = \mathbf x_k$ and
$\Delta_{k+1}=\gamma_{\mathrm{dec}}\Delta_k$ with
$\gamma_{\mathrm{dec}}\in(0,1)$.  Since the ratio $\rho_k$ is computed
solely from the Cauchy step for $m_k$, the same contradiction argument
as in the case $k\in\mathcal A\cup(\mathcal U \backslash \mathcal U^{sw})$ applies verbatim: assuming $\|\mathbf{g}_k\|>C_0\Delta_k$ would imply
$\rho_k\ge \eta_1$, contradicting $k\in\mathcal U^{sw}\subset
\mathcal U$.  Hence
\[
\|\mathbf{g}_k\|\le C_0\Delta_k,
\qquad
\|\nabla f(\mathbf{x}_k)\|\le (\kappa_{eg}+C_0)\Delta_k .
\]

After this unsuccessful TR iteration, a DS step is attempted from
$\mathbf x_{k+1}=\mathbf x_k$. If DS succeeds and 
let $\mathbf s_{k+1}^{\mathrm{DS}}$ be the distance from $\mathbf x_{k+1}$ to the new point $\mathbf x_{k+2}$, 
then it generates $\mathbf x_{k+2} = \mathbf x_{k+1}+ \mathbf s_{k+1}^{\mathrm{DS}}$ while keeping the radius unchanged,
$\Delta_{k+2}=\Delta_{k+1}$. The geometry correction then ensures that
$Y_{k+2}\supseteq Y_k\cup\{\mathbf x_{k+2}\}$ is $\Lambda$-poised in
$B(\mathbf x_{k+2},\Delta_{k+2})$, so $m_{k+2}$ is fully linear on this ball
with the same constants $\kappa_{ef},\kappa_{eg},\kappa_{bhm}$.

At iteration $k+2$, a Cauchy step $\mathbf s_{k+2}$ in (\ref{eq:subproblem_TR}) is computed for
$m_{k+2}$. Whenever the radius is reduced at iteration $k+2$ (i.e.\ 
$k+2\in\mathcal C^r\cup\mathcal A\cup\mathcal U$), applying the same
contradiction argument to $m_{k+2}$ yields
\[
\|\mathbf g_{k+2}\|\le C_0\Delta_{k+2},
\qquad
\|\nabla f(\mathbf x_{k+2})\|\le (\kappa_{eg}+C_0)\Delta_{k+2}.
\]

For clarity, the associated radii, iterates, and step bounds 
for switching from DS to TR 
are
\[
\Delta_{k+1}=\gamma_{\mathrm{dec}}\Delta_k,\qquad
\Delta_{k+2}=\Delta_{k+1},
\]
\[
\mathbf x_{k+1}= \mathbf x_k,\qquad
\mathbf x_{k+2}=\mathbf x_{k+1}+ \mathbf s_{k+1}^{\mathrm{DS}},\qquad
\mathbf x_{k+3}=\mathbf x_{k+2}+ \mathbf s_{k+2},
\]
\[
\|\mathbf x_{k+2}-\mathbf x_{k+1}\| = \| \mathbf s_{k+1}^{\mathrm{DS}}\| \;\le\;\delta_{k+1}^0\,\vartheta
=\vartheta \Delta_{k+1}
=\vartheta \gamma_{\mathrm{dec}} \Delta_k,
\]
\[
f(\mathbf x_{k+3})\le f(\mathbf x_{k+2})< f(\mathbf x_{k+1})=f(\mathbf x_k),
\]
where $K_{\mathrm{ds}}$ is at most the number of evaluations performed per DS iteration, and parameter $\vartheta<\infty$. 
Since DS does not modify the trust-region radius and the number of
search/poll steps is uniformly bounded by $K_{\mathrm{ds}}$, the above shows
that the iterates remain within $O(\Delta_{k+1})$ of $\mathbf x_k$, and the
fully linear model property is preserved.  Hence $\|\mathbf{g}_{k+i}\|\le
C_0\Delta_{k+i}$ whenever the radius is reduced, and therefore
\eqref{eq:gk_upper_bdd} holds in all cases.
\end{itemize}
\end{proof}

\begin{lemma}
    Suppose Assumptions \ref{ass:smoothness} 
    and \ref{ass:cauchy} hold. Then 
    \begin{align*}
        \liminf_{k\to +\infty} ||\nabla f(\mathbf{x}_k)||=0.
    \end{align*}
\end{lemma}

\begin{proof}
    It follows from Lemma \ref{lemma:TR_radius} that there exists an infinite subsequence of iterations $\{k_i\} \in \mathcal{C}^r \bigcup \mathcal{A} \bigcup \mathcal{U}$ such that $\Delta_{k_i}$ is decreased. Then according to Lemma \ref{lemma:gk_fk_upper_bdd}, for any sufficiently large $k_i$ we have $\lim_{k_i \to +\infty} ||\nabla f(\mathbf{x}_{k_i})||= \liminf_{k\to +\infty} ||\nabla f(\mathbf{x}_k)||=0$.
\end{proof}

\begin{theorem}
    Suppose Assumptions \ref{ass:smoothness} 
    and \ref{ass:cauchy} hold. Then 
    \begin{align*}
        \lim_{k\to +\infty} ||\nabla f(\mathbf{x}_k)||=0.
    \end{align*}
\end{theorem}

\begin{proof}
    If $|\mathcal{S}|<+\infty$, then the result is shown in Lemma \ref{lemma:finite_successful_iter}. Hence, we consider $|\mathcal{S}|=+\infty$. By contradiction, assume there exists a subsequence $\{ k_i \} \in \mathcal{S}$ and some $\epsilon_0>0$, such that 
    \begin{align} \label{eq:gradient_fx_assumption}
        ||\nabla f(\mathbf{x}_{k_i})|| \geq  \epsilon_0 >0.
    \end{align}
    Then by Lemma~\ref{lemma:fullylinear_success}, $||\mathbf{g}_{k_i}||\geq \epsilon>0$ for some $\epsilon$. Without loss of generality, we suppose that
    \begin{align*}
        \epsilon <\min \left\{ \epsilon_{low}, \frac{\epsilon_0}{2+\kappa_{eg}\mu} \right\}.
    \end{align*}
Let $j_i > k_i$, where $j_i \in \mathcal{C}^r \cup \mathcal{A} \cup \mathcal{U}$, be the first iteration such that $\|\mathbf{g}_{j_i}\| < \epsilon$,
which must exist from Lemma~\ref{lemma:gk_fk_upper_bdd}. In other words, there exists a subsequence $\{k_i\}$ such that
\[
\|\mathbf{g}_n\| \ge \epsilon \quad \text{for } k_i \le n < j_i, 
\qquad \text{and} \qquad 
\|\mathbf{g}_{j_i}\| < \epsilon .
\]

Consider the set 
\[
\mathcal{K} := \bigcup_{i\ge 0} \{ n \in \mathbb{N}_0 : k_i \le n < j_i \}.
\]
For each iteration $n \in \mathcal{K} \cap \mathcal{S}$, it holds that
\begin{align}\label{eq:fx_decrease}
f(\mathbf{x}_n) - f(\mathbf{x}_{n+1}) 
&\ge \eta_1 \left( m_n(\mathbf{x}_n) - m_n(\mathbf{x}_n + \mathbf{s}_n) \right) \nonumber \\
&\ge \frac{\eta_1}{2} \|\mathbf{g}_n\| 
\min \left\{ \Delta_n, \frac{\|\mathbf{g}_n\|}{\|H_n\|} \right\} \nonumber \\
&\ge \frac{\eta_1 \epsilon}{2} 
\min \left\{ \Delta_n, \frac{\|\mathbf{g}_n\|}{\|H_n\|} \right\} > 0 .
\end{align}
Thus, for any $n \in \mathcal{K} \cap \mathcal{S}$ sufficiently large, we have
\[
\Delta_n \le \frac{2\big(f(\mathbf{x}_n) - f(\mathbf{x}_{n+1})\big)}{\eta_1 \epsilon}.
\]
Since $\|\mathbf{g}_n\| \ge \epsilon$ by assumption and $\Delta_n \to 0$ by Lemma~\ref{lemma:TR_radius}, Lemma~\ref{lemma:finite_successful_iter} implies that for all sufficiently large $n \in \mathcal{K}$, either $n \in \mathcal{S}$ with $m_n$ fully linear, or $n \in \mathcal{MI}$. The iterates $\mathbf{x}_n$ are updated during successful steps and remain unchanged during model-improving steps. Therefore, for all sufficiently large $i$, we obtain
\begin{align}\label{eq:xk_upper_bdd}
\|\mathbf{x}_{k_i} - \mathbf{x}_{j_i}\|
&\le \sum_{\substack{n = k_i \\ n \in \mathcal{K} \cap \mathcal{S}}}^{j_i - 1}
\|\mathbf{x}_n - \mathbf{x}_{n+1}\| \le \sum_{\substack{n = k_i \\ n \in \mathcal{K} \cap \mathcal{S}}}^{j_i - 1}
\Delta_n \le \sum_{\substack{n = k_i \\ n \in \mathcal{K} \cap \mathcal{S}}}^{j_i - 1}
\frac{2\big(f(\mathbf{x}_n) - f(\mathbf{x}_{n+1})\big)}{\eta_1 \epsilon}. 
\end{align}
This further implies that
\begin{align}
\|\mathbf{x}_{k_i} - \mathbf{x}_{j_i}\| = \frac{2\left(f(\mathbf{x}_{k_i}) - f(\mathbf{x}_{j_i})\right)}{\eta_1 \epsilon}.
\end{align}
By Assumption~\ref{ass:smoothness} and \eqref{eq:fx_decrease}, the sequence 
$\{f(\mathbf{x}_n) : n \in \mathcal{K}\}$ is bounded and monotone decreasing. Hence, the right-hand side of \eqref{eq:xk_upper_bdd} converges to zero, i.e., $\left\{ f(\mathbf{x}_{k_i}) - f(\mathbf{x}_{j_i}) \right\} \to 0.$
Consequently,
\[
\lim_{i \to \infty} \|\mathbf{x}_{k_i} - \mathbf{x}_{j_i}\| = 0 .
\]  
Consider
\begin{align*}
        ||\nabla f(\mathbf{x}_{k_i})|| &\leq ||\nabla f(\mathbf{x}_{k_i}) -\nabla f(\mathbf{x}_{j_i})|| + ||\nabla f(\mathbf{x}_{j_i})-\mathbf{g}_{j_i}|| + ||\mathbf{g}_{j_i}|| \\
        &\leq L_{\nabla f} ||\mathbf{x}_{k_i}-\mathbf{x}_{j_i}|| + \kappa_{eg}\Delta_{j_i} + \epsilon \\
        &\leq \epsilon + \kappa_{eg}\mu \epsilon + \epsilon \hspace{15mm} \text{(by $\Delta_{j_i}<\mu||\mathbf{g}_{j_i}|| < \mu\epsilon$)} \\
        &= (2+\kappa_{eg}\mu)\epsilon <\epsilon_0,
    \end{align*}
    which contradicts (\ref{eq:gradient_fx_assumption}). This emphasizes our claim and completes the proof. 
\end{proof}

\section{Numerical Evaluation}
\label{sec:numeric}

This section presents a comprehensive evaluation of the numerical performance of several optimization algorithms on both classical simulation-based benchmarks and representative machine learning tasks. We compare two broad categories of methods: conventional state-of-the-art derivative-free solvers and our proposed hybrid optimization schemes. The baseline solvers include a trust-region method implemented via Py-BOBYQA~\cite{cartis2019improving, cartis2022escaping}, a Bayesian optimization method based on the Efficient Global Optimization framework~\cite{jones1998efficient} available in the Surrogate Model Toolbox, and a direct search method implemented using the Mesh Adaptive Direct Search algorithm~\cite{audet2006mesh}. These will be referred to as basic TR, basic BO, and basic DS, respectively.

From the perspective of our proposed methodologies, we examine two hybrid strategies: Trust Region combined with Direct Search (TR-DS, Algorithm~\ref{alg:alg_1}) and Bayesian Optimization combined with Direct Search (BO-DS, Algorithm~\cite{lion24-aswin}). All experiments were implemented in Python $3.7$, and executed on a workstation with an NVIDIA RTX $3080$ GPU and an Intel Core i5-12600K processor, ensuring computational consistency across all problem classes. This section is organized into two subsections, each corresponding to one of the problem classes under investigation.


\subsection{CUTEr DFO Set}
The CUTEr testbed~\cite{more2009benchmarking} provides a diverse suite of constrained and unconstrained optimization problems, varying in both complexity and dimensionality, making it well-suited for evaluating the robustness of our algorithms. We select $13$ benchmark problems from this suite to ensure a broad spectrum of challenges. As summarized in Table~\ref{table:num_dim_cuter_problems}, the problem dimensions range from $5$ to $30$, with a median dimensionality of $19$.

\begin{table}[ht]
\centering
\caption{Dimensions of the benchmark problems.}
\label{table:num_dim_cuter_problems}
\begin{tabular}{|c|c|c|c|c|c|c|}
\hline
Number of dimensions & 5 & 8 & 15 & 20 & 25 & 30  \\
\hline
Number of problems & 1 & 1 & 3 & 4 & 2 & 2 \\
\hline
\end{tabular}
\end{table}


\subsubsection{Metrics for benchmarking}
Since the problems in our consideration are simulation-based, standard termination criteria used in gradient-based methods may not be directly applicable. Instead, we adopt the well-established performance profiles proposed in~\cite{more2009benchmarking} for our analysis. 

Performance profiles are estimated on top of some convergence criteria. The convergence criteria can pertain to function value, computational time, etc. We also note that performance profiles apply to problems with gradients. For example, if defined with respect to objective function values, they take the following form. 
\begin{align} \label{eq:convergence_test}
    f(x) \leq f_L +\tau \left(f(x_0) - f_L \right),
\end{align}
Note that $\tau \in (0,1]$ is a tolerance that can be specified by users. This method gives a direct quantitative measure of the progress made by any solver relative to the best-known solution. By choosing an appropriate $\tau$, the accuracy level can be adjusted to satisfy different application requirements. 

We employ two well-known metrics, i.e., performance profiles and data profiles, to assess and compare the performance of various algorithms in the benchmark problems.


\textbf{Performance profiles:} Dolan \textit{et al.}~\cite{dolan2002benchmarking} introduced \emph{performance profiles} as a reliable and informative method for comparing the efficiency of multiple solvers across a set of benchmark problems. This approach addresses the inherent difficulty in evaluating solver performance, particularly when no single solver consistently dominates across all test cases.

Given a set of benchmark problems $\mathcal{P}$ and a collection of solvers $\mathcal{S}$, the \emph{performance ratio} $r_{p,s}$ for solver $s \in \mathcal{S}$ on problem $p \in \mathcal{P}$ is defined as:
\begin{align*}
    r_{p,s} = \frac{t_{p,s}}{\min \{ t_{p,s'} : s' \in \mathcal{S} \}},
\end{align*}
where $t_{p,s}$ denotes the computational cost (e.g., runtime, number of iterations, or function evaluations) incurred by solver $s$ on problem $p$. By definition, the best-performing solver for each problem has a ratio of $r_{p,s} = 1$, while larger values of $r_{p,s}$ indicate poorer relative performance.

The \emph{performance profile} $\varrho_s(\alpha)$ of the solver $s$ is given by:
\begin{align} \label{eq:performance_profiles}
    \varrho_s(\alpha) = \frac{1}{|\mathcal{P}|} \left| \left\{ p \in \mathcal{P} : r_{p,s} \leq \alpha \right\} \right|,
\end{align}
where $|\mathcal{P}|$ denotes the cardinality of the problem set. The function $\varrho_s(\alpha)$ represents the fraction of problems for which the solver $s$ performs within a factor $\alpha \in \mathbb{R}$ of the best solver. In particular, $\varrho_s(1)$ measures the proportion of problems where the solver $s$ achieves the best performance. A higher value of $\varrho_s(\alpha)$ reflects superior performance across a broader range of problems.

\textbf{Data Profiles.} 
As an alternative benchmarking approach, Móré and Wild~\cite{more2009benchmarking} proposed \emph{data profiles}, which evaluate a solver’s ability to solve problems within a prescribed computational budget. Unlike raw timing measures, data profiles explicitly account for problem dimensionality.

To define this, consider the \emph{unit cost} for the solver $s$ on problem $p$:
\begin{align*}
    \text{ucost}_{p,s} := \frac{t_{p,s}}{n_p + 1},
\end{align*}
where $n_p$ denotes the number of variables (i.e., the dimensionality) of problem $p$. The \emph{data profile} $dp_s(\alpha)$ is then defined as
\begin{align} \label{eq:data_profiles}
    dp_s(\alpha) = \frac{1}{|\mathcal{P}|} \left| \left\{ p \in \mathcal{P} : \text{ucost}_{p,s} \leq \alpha \right\} \right|.
\end{align}

By normalizing the computational cost with respect to problem size, data profiles allow fairer comparisons between problems of different scales. Similar to performance profiles, a higher value of $dp_s(\alpha)$ indicates that the solver $s$ is effective across a larger portion of the problem set under a given budget.

\subsubsection{Results for benchmarking problems}
In our experiments, the quantity $t_{p,s}$ in equations~\eqref{eq:performance_profiles} and~\eqref{eq:data_profiles} consistently refers to the computational time.

Figure~\ref{fig:performance_profile} presents the performance profiles of the solvers evaluated at three convergence tolerances, namely $\tau \in \{10^{-1}, 10^{-3}, 10^{-5}\}$. A smaller value of $\tau$ corresponds to a stricter convergence criterion. Overall, the TR-DS and BO-DS solvers demonstrate strong and consistent performance across these settings.

\begin{figure*}[htbp] 
  \centering
  \begin{subfigure}{.325\textwidth}
    \centering
    \includegraphics[width=\linewidth]{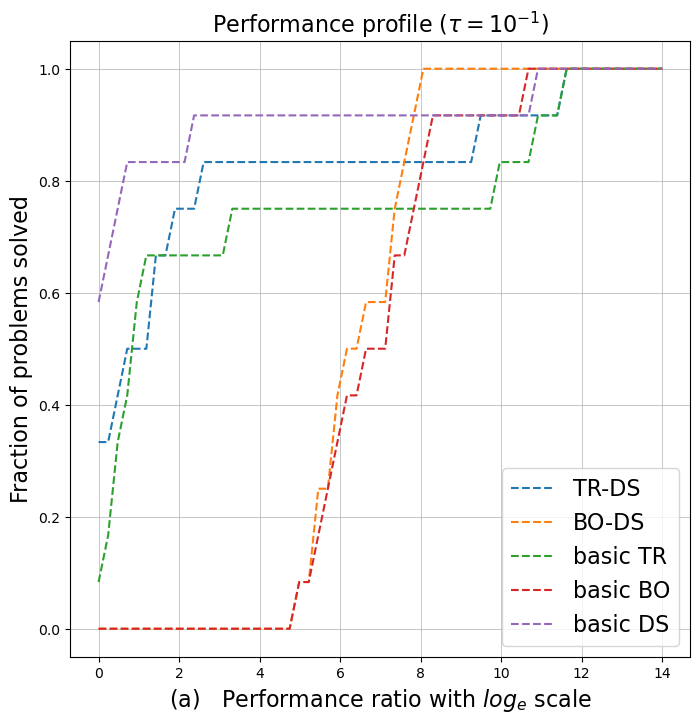}
  \end{subfigure}%
\begin{subfigure}{.325\textwidth}
    \centering
    \includegraphics[width=\linewidth]{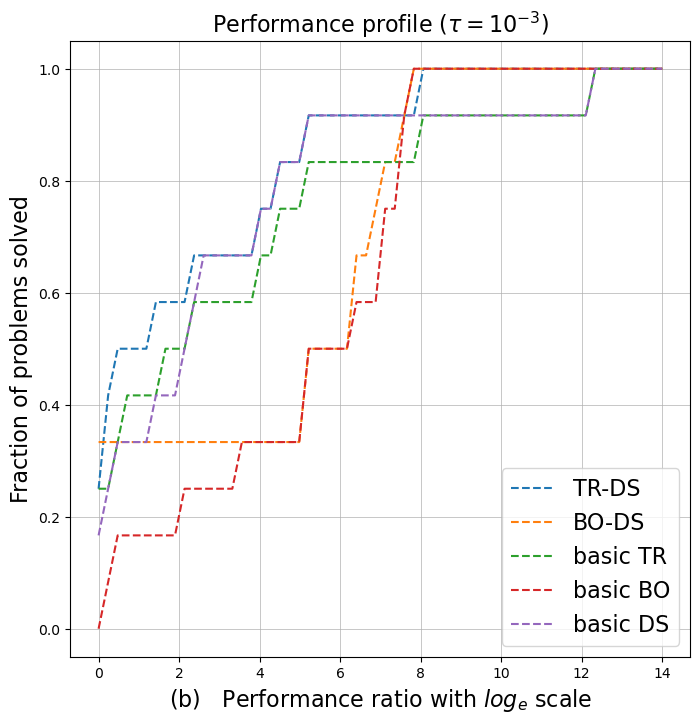}
  \end{subfigure}
  \begin{subfigure}{.325\textwidth}
    \centering
    \includegraphics[width=\linewidth]{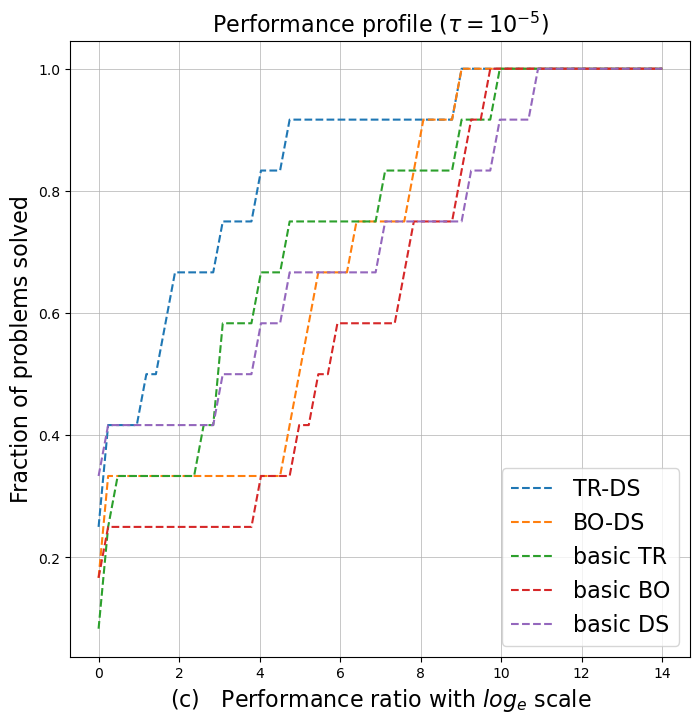}
  \end{subfigure}  
  \caption{Performance profiles with different choices of $\tau$ comparing the hybrid solvers (TR-DS and BO-DS) against baseline methods (basic TR, BO, and DS) on CUTEr benchmark set. 
  Note that TR-DS and BO-DS show consistently high performance across all tolerance levels, solving a higher fraction of problems at lower performance ratios compared to basic TR and BO methods.}
  \label{fig:performance_profile}
\end{figure*}

For $\tau = 10^{-1}$, the TR-DS and basic DS solvers perform particularly well, solving approximately $81\%$ and $90\%$ of the problems, respectively, within a performance ratio of $r_{p,s} = 3$. Although the BO-DS and basic BO solvers initially lag behind, their performance improves significantly as the allowed performance ratio increases (i.e., for $r_{p,s} \geq 6$). Notably, BO-DS is the first solver to reach full coverage, solving $100\%$ of the test problems.

When the tolerance is tightened to $\tau = 10^{-3}$, all solver curves shift to the right, reflecting the increased difficulty of achieving convergence. Nonetheless, TR-DS and basic DS remain strong performers at lower performance ratios. While BO-DS and basic BO underperform initially, they once again surpass the others at higher ratios, confirming their effectiveness when given sufficient computational leeway.

At the most stringent setting, $\tau = 10^{-5}$, nearly all solvers reach full problem coverage by $r_{p,s} = 10$. TR-DS notably outperforms the other methods across almost the entire range of performance ratios. Moreover, both TR-DS and BO-DS are the first solvers to achieve $100\%$ success at the same ratio threshold.

Taken together, the results in Figure~\ref{fig:performance_profile} highlight that although all solvers are ultimately capable of solving most benchmark problems, their efficiency—particularly at tighter tolerances—varies substantially. Among them, TR-DS stands out for its robust and consistently high performance, irrespective of the specified accuracy requirement.
\begin{figure*}[htbp] 
  \centering
  \begin{subfigure}{.325\textwidth}
    \centering
    \includegraphics[width=\linewidth]{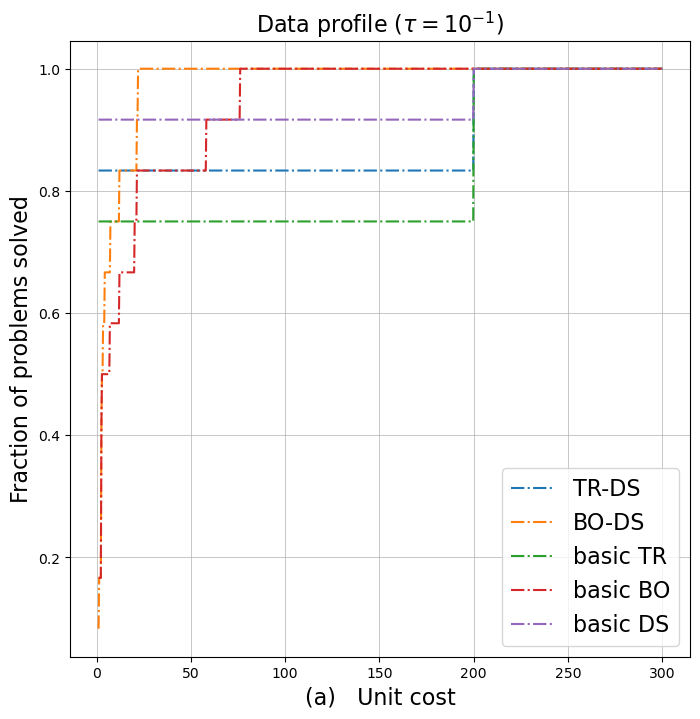}
  \end{subfigure}%
\begin{subfigure}{.325\textwidth}
    \centering
    \includegraphics[width=\linewidth]{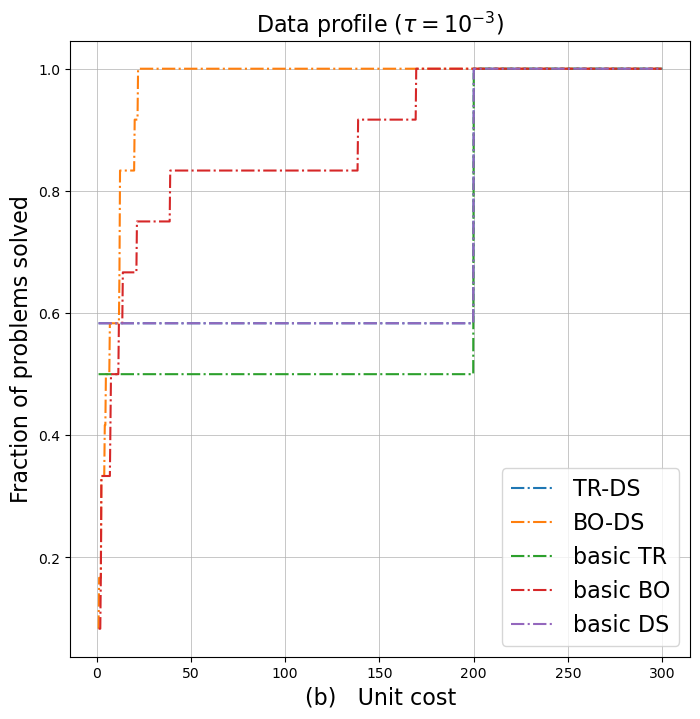}
  \end{subfigure}
  \begin{subfigure}{.325\textwidth}
    \centering
    \includegraphics[width=\linewidth]{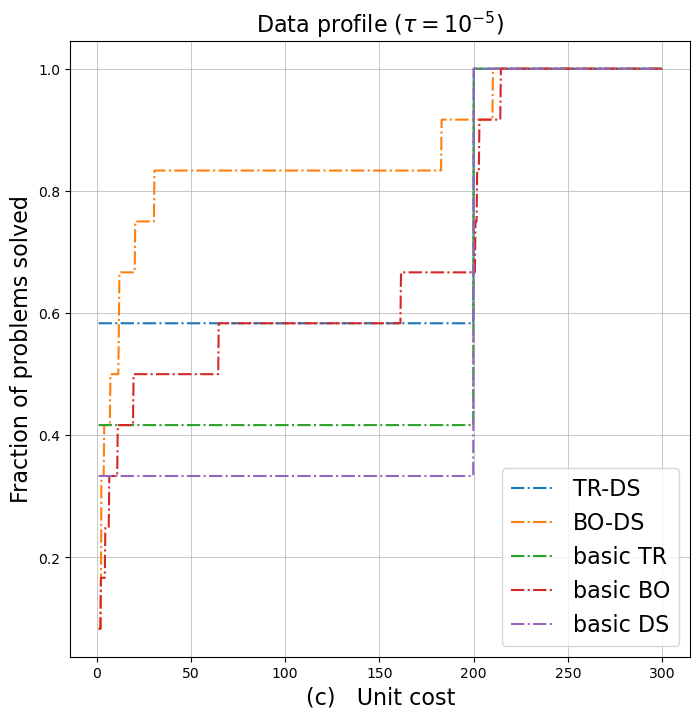}
  \end{subfigure}  
  \caption{Data profiles with different choices of $\tau$ evaluate the computational efficiency of the hybrid solver (TR-DS and BO-DS) and the baseline solver on the CUTEr benchmark set. 
  The BO-DS method shows strong performance, which solves the largest fraction of problems within small unit costs. Furthermore, TR-DS consistently outperforms the basic TR method, indicating that the switching mechanism effectively enhances the efficiency of the basic TR framework.}
  \label{fig:data_profile}
\end{figure*}

Figure~\ref{fig:data_profile} displays the data profiles corresponding to the same three convergence tolerances as in Figure~\ref{fig:performance_profile}. Notably, in Figure~\ref{fig:data_profile}(b), the curves for TR-DS and basic DS completely overlap, indicating identical performance under that setting.

Across all values of $\tau$, BO-DS demonstrates strong performance in Figure~\ref{fig:data_profile}, solving a substantial fraction of problems even at low unit costs. This stands in contrast to the performance profile results shown in Figure~\ref{fig:performance_profile}, where BO-based methods performed less favorably at lower performance ratios. A possible explanation is that Bayesian optimization is particularly effective for problems with moderate dimensionality (approximately 15 to 40 variables).

Consistent with the findings in Figure~\ref{fig:performance_profile}, both TR-DS and BO-DS outperform their basic counterparts (basic TR and basic BO) across all tolerance levels. Furthermore, as $\tau$ decreases, the curves in both figures shift downward and to the right, reflecting the increased difficulty of achieving stricter convergence.

Among the solvers, basic DS appears to be more sensitive to the tolerance level, whereas TR-DS and BO-DS exhibit relatively stable performance across all $\tau$. This robustness suggests that TR-DS and BO-DS generalize well and are well-suited for applications requiring either high solution accuracy or efficient use of computational resources.

\subsection{Machine learning (ML) applications}

In this study, we use weighted sum methods when dealing with multi-objective problems. Specifically, in the context of machine learning, the multi-objective hyperparameter tuning problem can be solved by scalarizing all the objectives into one objective and then applying the model-based DFO strategy discussed earlier to find the Pareto optimal solution. In the following, we will tackle hyperparameter optimization for regression and multi-class classification tasks.

\subsubsection{Data and its processing} 
We employ four real-world datasets sourced from two locations in Germany: Berlin and Stuttgart. These datasets contain two types of energy data: solar and wind energy. For solar energy, we formulate a regression problem, whereas for wind energy, we adopt a multi-class classification framework as in previous works in literature \cite{wind-complete-16, solar-21-thorough}. The feature data for these tasks was obtained from the National Renewable Energy Laboratory (NREL) \cite{NREL}, which includes structured information on weather factors such as temperature, humidity, and wind speed, amounting to twenty-three input features across all datasets. The generational energy data, which serves as the target variable, was sourced from Netztransparenz \cite{Netztransparenz} and covers the period from January $1$st, $2018$, to December $31$st, $2019$.

The data processing for the solar energy regression task involves refining the dataset by excluding periods when solar power generation is not possible, specifically nighttime hours from $10$ p.m. to $5$ a.m., ensuring only non-zero production data points are included. This preprocessing step results in approximately $36,000$ data points for each solar dataset. 

For the wind energy multi-class classification task, we standardize all features (excluding the target variable) using Z-scores. The actual wind power generation data is divided into five nearly equal-sized classes to maintain a balanced dataset, resulting in around $70,000$ data points per dataset. The choice of five classes is made to avoid imbalances that could adversely affect the classification performance. 

Both types of datasets are split into training, test, and validation sets using an $80\% : 10\% : 10\%$ ratio. This preparation ensures that the models trained on these datasets can make predictions and generalizations, thereby showing the efficacy of our method in real-world scenarios.

\subsubsection{Metrics for ML problems}
To comprehensively evaluate the performance of our algorithm on the given datasets, we considered three types of metrics: accuracy, bias, and complexity. These metrics are chosen to ensure an assessment of the models in terms of prediction performance, fairness, and structural simplicity.

\textbf{Accuracy metrics:} For the solar energy regression problem, the primary accuracy metrics used are \textit{Mean Squared Error (MSE)} and \textit{Mean Bias Error (MBE)}. MSE measures the average of the squares of the errors, while MBE quantifies the average bias in the predictions. These are calculated as:
\begin{align*}
    MSE(y) = \frac{1}{N} \sum_{i=1}^N \left( yd_i \right)^2, \quad
    MBE(y) = \frac{1}{N} \left|\sum_{i=1}^N yd_i \right|,
\end{align*}
where $N$ is the number of samples, $y^{\text{true}}$ and $y^{\text{pred}}$ are the true and predicted labels, respectively, and $yd = y^{\text{true}} - y^{\text{pred}}$.

For the wind energy multi-class classification task, we use \textit{Accuracy Error (ACE)}, which measures the proportion of incorrectly classified instances:
\begin{align*}
    ACE(y) = 1 - \frac{1}{N}\sum_{i=1}^N \left|yd_i\right|.
\end{align*}

\textbf{Bias metrics:} To evaluate fairness in the models, particularly for the wind energy multi-class classification task, we use \textit{Demographic Parity Difference (DPD)} and \textit{Symmetric Distance Error (SDE)}. DPD measures the difference in average predictions across groups defined by a binary attribute $a_i \in \{0, 1\}$:
\begin{align*}
    A_0 &:= \{i \in \{1,\dots,N\} \mid a_i = 0\}, \\
    A_1 &:= \{i \in \{1,\dots,N\} \mid a_i = 1\},
\end{align*}
where $A_0$ and $A_1$ contain $N_0$ and $N_1$ samples respectively, and $N = N_0 + N_1$. SDE captures the discrepancy between false negative rates (FNR) and false positive rates (FPR) across $H$ binary features:
\begin{align*}
    SDE = \frac{1}{H} \sum_{c=1}^H \left| \Delta FNR(c) - \Delta FPR(c) \right|.
\end{align*}

\textbf{Model complexity:} Model complexity is assessed differently based on the architecture. For simpler models like K-Nearest Neighbors (KNN), we focus on accuracy and bias metrics. For more complex models such as ResNet and XGBoost, we include structural metrics.

In ResNet models, we use \textit{Weight Sparsity (WSpar)} to denote the proportion of zero or near-zero weights. To quantify complexity, we minimize \textit{Weight Density (WDens)}, defined as:
\begin{align*}
    WDens = 1 - WSpar.
\end{align*}

In XGBoost models, complexity is measured by the \textit{average depth (ADep)} of the trees. Deeper trees indicate higher model complexity. Table~\ref{table:metrics_for_different_models} summarizes the metrics used for different model and problem types.

\begin{table}[htbp]
\centering
\caption{Metrics to minimize in different models.}
\label{table:metrics_for_different_models}
\begin{tabular}{|l|l|l|}
\hline
Problem type & Model & Metrics  \\
\hline
\multirow{3}{4em}{Solar regression} & ResNet & MSE, MBE, WDens \\
 & XGBoost & MSE, MBE, ADep \\
 & KNN & MSE, MBE \\
\hline
\multirow{3}{4em}{Wind classification} & ResNet & ACE, DPD, WDens \\
 & XGBoost & ACE, DPD, ADep \\
 & KNN & ACE, SDE \\
\hline
\end{tabular}
\end{table}

\textbf{Hypervolume indicator:} In multi-objective optimization, the \textit{hypervolume indicator}~\cite{zitzler2002multiobjective,kannan2021hyperaspo} is a widely accepted metric that assesses the quality of a set of solutions by measuring the volume of the objective space dominated by the Pareto front relative to a reference point. A higher hypervolume signifies better and more diverse solutions.

\subsubsection{Results for machine learning problems}

We address multi-objective optimization by converting it into a sequence of scalarized single-objective problems using a weighted-sum method. The weights are sampled from a uniform distribution $w\in \mathbb{W}^{\textrm{uniform}}$ such that $\sum_{i=1}^m w_i = 1$. The goal is to demonstrate the efficacy of our hybrid algorithms (TR-DS and BO-DS) compared to their base solvers, rather than optimizing the scalarization process itself.

For bi-objective problems, we use five evenly spaced weight pairs:
\begin{align*}
    (w_1, w_2) \in \{(0.25i, 1 - 0.25i) \mid i \in \{0,1,2,3,4\}\}.
\end{align*}
For tri-objective cases:
\begin{align*}
    w^T = \frac{1}{3}(i_1, i_2, i_3), \quad i_1 + i_2 + i_3 = 3, \quad i_j \in \{0,1,2,3\}.
\end{align*}
We compute the hypervolume across the aggregate solution sets from all weight instances.

We evaluate the algorithms on ResNet, XGBoost, and KNN models, optimizing relevant hyperparameters. Table~\ref{table:HPO} lists the selected hyperparameters along with their lower and upper bounds. For KNN, $\varphi_m$ denotes the power parameter in the Minkowski metric.

\begin{table}[htbp]
\centering
\caption{Hyperparameters selected for optimization in each model.}
\label{table:HPO}
\begin{tabular}{|l|l|c|c|}
\hline
Model & Hyperparameters & lb & ub \\
\hline
\multirow{3}{*}{ResNet} & Number of residual blocks & 1 & 15 \\
 & Number of neurons & 8 & 128 \\
 & Dropout rate & 0.01 & 0.5 \\
\hline
\multirow{5}{*}{XGBoost} & Number of estimators & 1 & 50 \\
 & Max depth & 2 & 20 \\
 & Max leaves & 1 & 20 \\
 & Min child weight & 0.1 & 5 \\
 & Gamma & 0 & 0.5 \\
\hline
\multirow{2}{*}{KNN} & Number of neighbors & 1 & 20 \\
 & $\varphi_m$ (Minkowski) & 1 & 4 \\
\hline
\end{tabular}
\end{table}

\begin{figure*}[h] 
  \centering
  \begin{subfigure}{.325\textwidth}
    \centering
    \includegraphics[width=\linewidth]{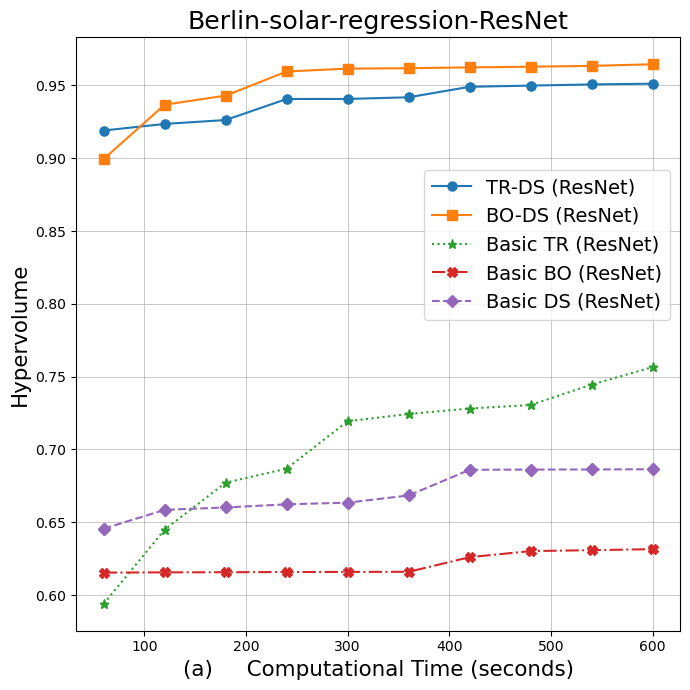}
  \end{subfigure}%
\begin{subfigure}{.325\textwidth}
    \centering
    \includegraphics[width=\linewidth]{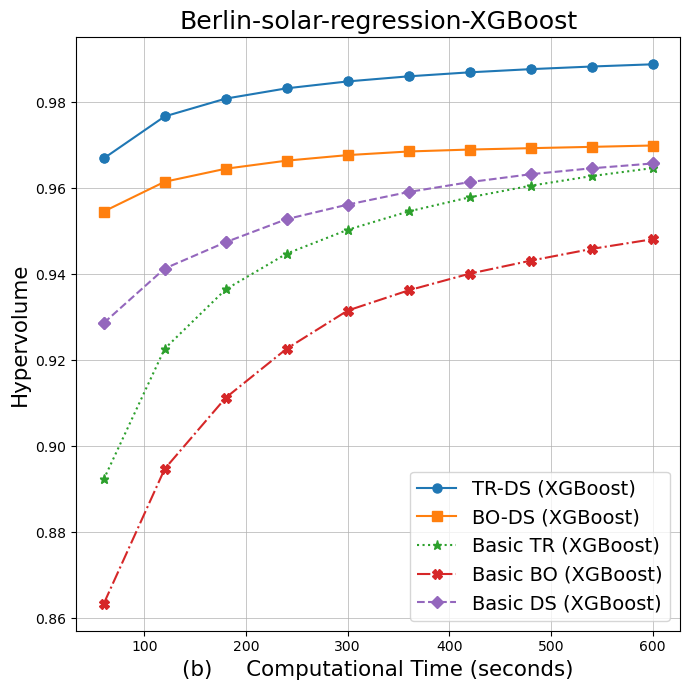}
  \end{subfigure}
  \begin{subfigure}{.325\textwidth}
    \centering
    \includegraphics[width=\linewidth]{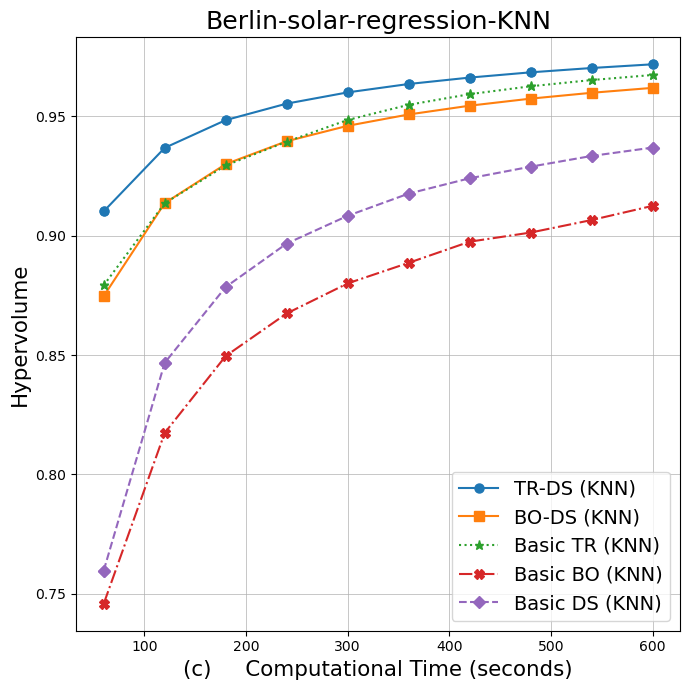}
  \end{subfigure}  \\
    \begin{subfigure}{.325\textwidth}
    \centering
    \includegraphics[width=\linewidth]{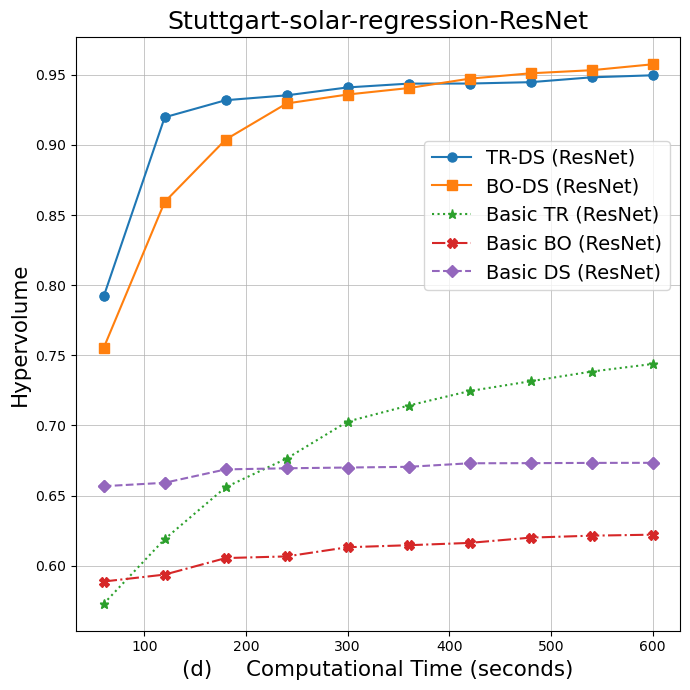}
  \end{subfigure}
\begin{subfigure}{.325\textwidth}
    \centering
    \includegraphics[width=\linewidth]{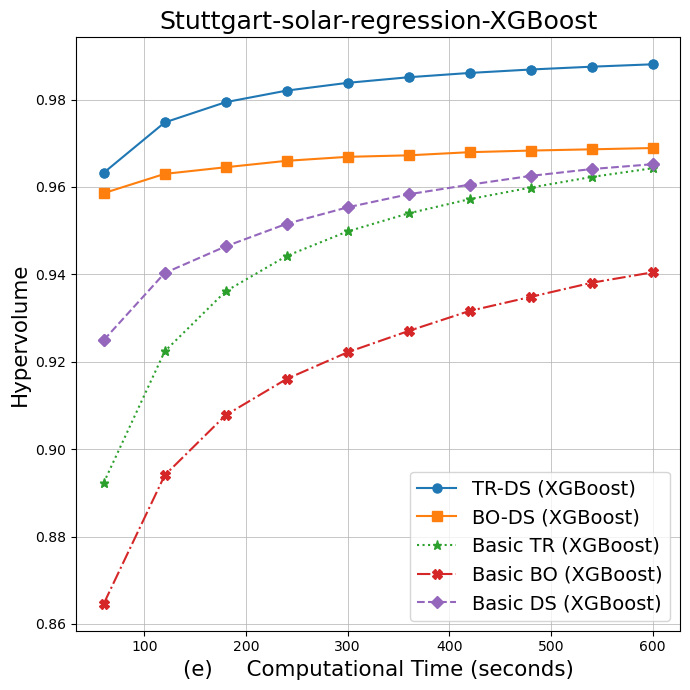}
  \end{subfigure}
 \begin{subfigure}{.325\textwidth}
    \centering
    \includegraphics[width=\linewidth]{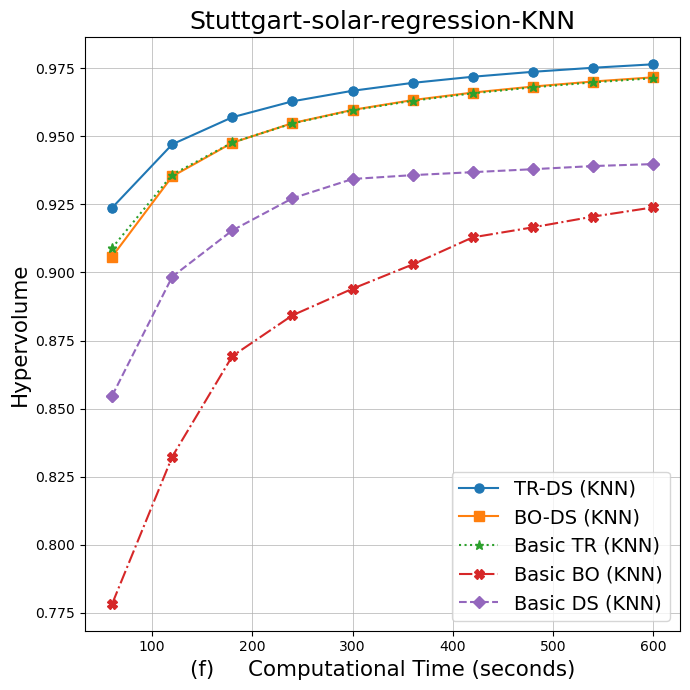}
  \end{subfigure}
  \caption{Hypervolume trajectories for solar regression tasks across ResNet, XGBoost, and KNN models. 
  TR-DS and BO-DS achieve higher hypervolume values more rapidly than the baseline solvers (basic TR and BO). Performance gaps become more noticeable in the complex ResNet and XGBoost models. 
  }
  \label{fig:resnet_XGB_KNN_regression}
\end{figure*}

Figure~\ref{fig:resnet_XGB_KNN_regression} shows the performance of different methods on regression tasks using ResNet, XGBoost, and KNN across two solar datasets (Berlin and Stuttgart). It is evident that TR-DS and BO-DS generally outperform their respective basic solver counterparts. This is particularly prominent in the ResNet models, where TR-DS and BO-DS achieve hypervolumes exceeding $0.95$, while the basic solvers remain below $0.75$.

The choice of machine learning model also influences optimization outcomes. For XGBoost, the performance gap narrows; however, TR-DS still yields the highest hypervolume. In the case of KNN, all solvers attain hypervolumes above $0.91$, suggesting that for simpler models, the advantage of combined methods is less pronounced. Nevertheless, TR-DS and BO-DS maintain consistent strengths, underscoring their robustness across varying model complexities. In contrast, basic solvers display more erratic behavior and are more sensitive to the choice of model.

In addition to the regression task, we evaluate solver performance on multi-class classification problems. Figure~\ref{fig:resnet_XGB_KNN_multiclass} illustrates the results of applying different solvers to ResNet, XGBoost, and KNN on the wind datasets.

\begin{figure*}[h] 
\centering
 \begin{subfigure}{.325\textwidth}
    \centering
    \includegraphics[width=\linewidth]{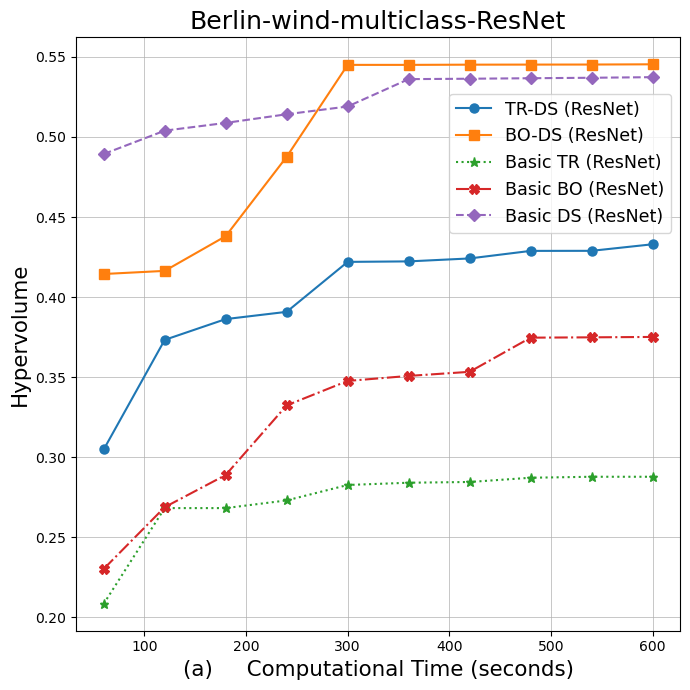}
  \end{subfigure}
   \begin{subfigure}{.325\textwidth}
    \centering
    \includegraphics[width=\linewidth]{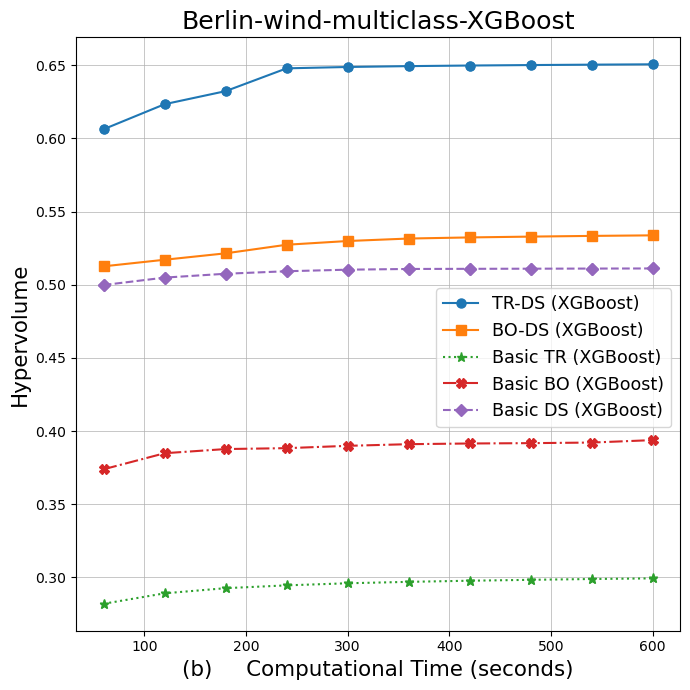}
  \end{subfigure}
 \begin{subfigure}{.325\textwidth}
    \centering
    \includegraphics[width=\linewidth]{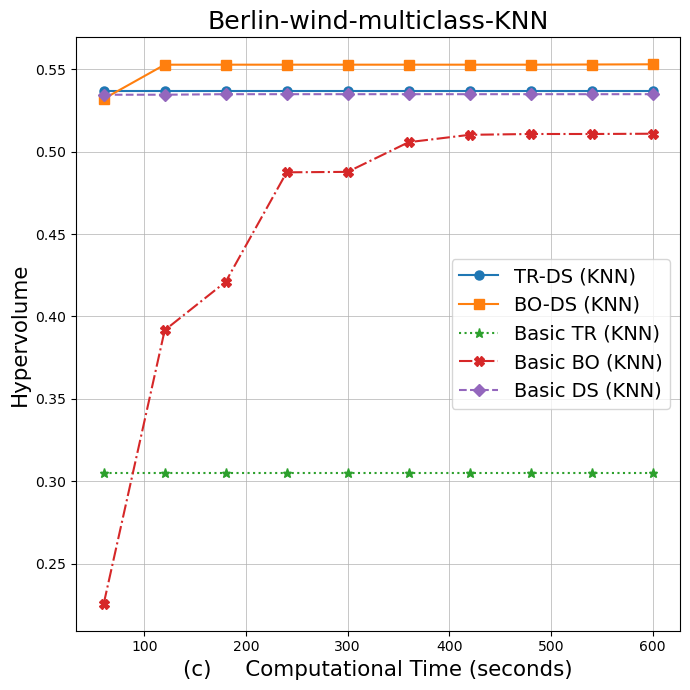}
  \end{subfigure}\\
 \begin{subfigure}{.325\textwidth}
    \centering
    \includegraphics[width=\linewidth]{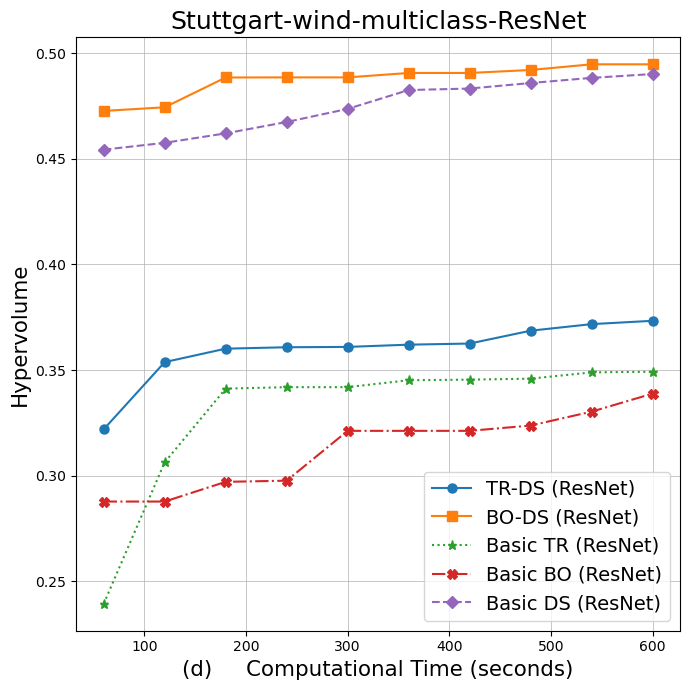}
  \end{subfigure} 
 \begin{subfigure}{.325\textwidth}
    \centering
    \includegraphics[width=\linewidth]{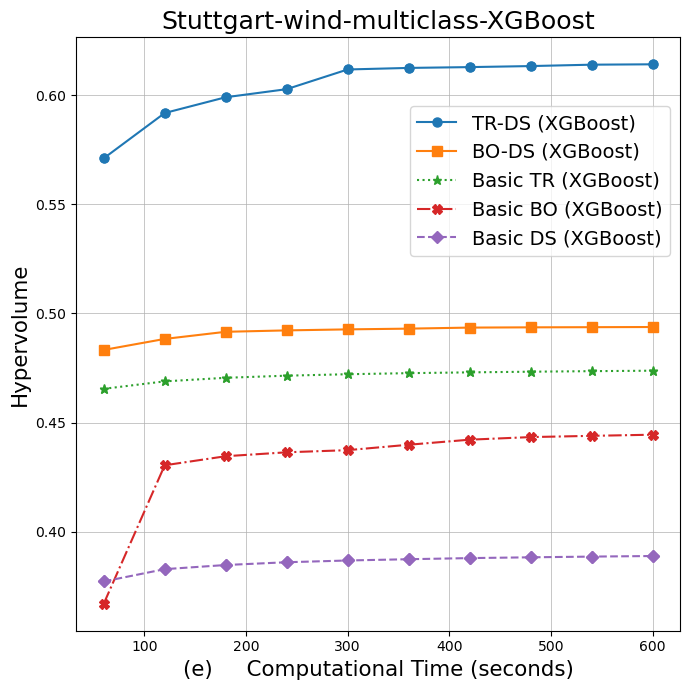}
  \end{subfigure} 
  \begin{subfigure}{.325\textwidth}
    \centering
    \includegraphics[width=\linewidth]{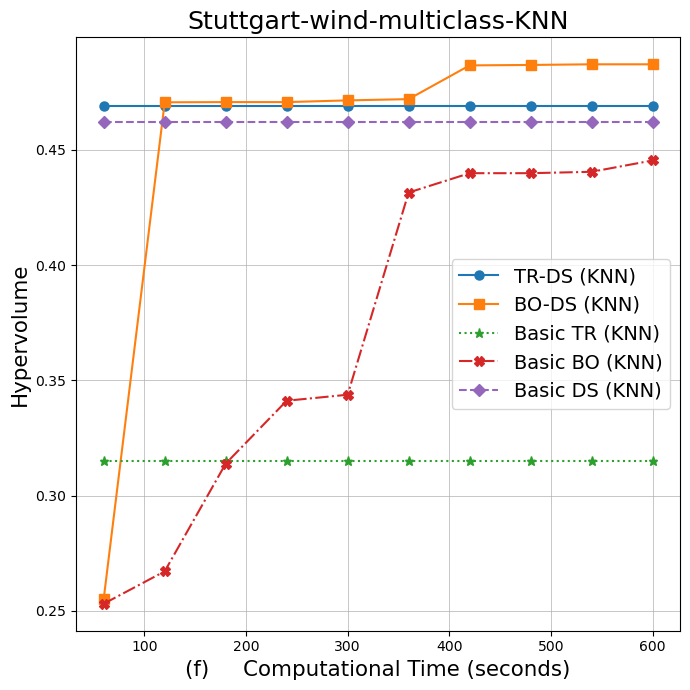}
  \end{subfigure}
  \caption{Hypervolume trajectories for wind energy multi-class classification tasks. 
  While absolute hypervolume values are lower than those of regression tasks due to classification metrics, the performance of TR-DS and BO-DS still outperforms the basic TR and BO solvers, with TR-DS presenting strong performance in XGBoost models.
}
  \label{fig:resnet_XGB_KNN_multiclass}
\end{figure*}

As observed in Figure~\ref{fig:resnet_XGB_KNN_multiclass}, the BO-DS solver generally exhibits strong performance in multi-class classification settings, particularly when used with ResNet and KNN. On ResNet, BO-DS achieves the highest hypervolume among all methods. TR-DS performs moderately but consistently surpasses both basic TR and basic BO. Interestingly, basic DS demonstrates robust performance and ranks as the second-best solver, only marginally behind BO-DS. On XGBoost, TR-DS achieves the highest hypervolume, indicating its effectiveness when used with decision tree-based models. Although BO-DS also performs well in this setting, the advantage of TR-DS is more apparent. For KNN, TR-DS, BO-DS, and basic DS all converge to similar high hypervolume values across both datasets. While basic BO improves steadily over time, its final performance remains lower than TR-DS, BO-DS, and basic DS when the time budget is exhausted.

Notably, the overall hypervolume values in multi-class classification tasks are significantly lower than those in the regression problems. This discrepancy is partly due to the use of different evaluation metrics tailored to each task.

From Figures~\ref{fig:resnet_XGB_KNN_regression} and~\ref{fig:resnet_XGB_KNN_multiclass}, it is clear that integrating direct search mechanisms significantly enhances the performance of basic solvers across various models and datasets. Specifically, TR-DS and BO-DS consistently outperform their respective base algorithms, highlighting the critical role of the DS mechanism in identifying superior solutions.

\section{Conclusion and future research} \label{sec:conclusion}

In this paper, we proposed hybrid model-based and search-driven methods for derivative-free optimization (DFO), with a central theoretical contribution being the global convergence guarantee for the TR-DS method. Our numerical experiments further support its practical effectiveness. We benchmarked TR-DS against classical trust-region (TR) methods and direct search (DS) methods, Bayesian optimization (BO), and a recent BO-DS hybrid approach~\cite{lion24-aswin} across standard DFO test suites and real-world tasks.

The results consistently indicate that TR-DS yields higher-quality solutions within tight computational budgets when compared to its basic counterparts. Moreover, in the context of multi-objective hyperparameter optimization (HPO) for machine learning, TR-DS achieves larger hypervolumes with significantly reduced computation time relative to basic TR, DS, and BO algorithms.

A promising direction for future research lies in developing a formal convergence analysis for the BO-DS approach, which emerged as a strong empirical performer in our study. Additionally, extending the TR-DS framework to effectively handle constrained DFO problems presents another valuable avenue for algorithmic enhancement.

\paragraph{Acknowledgments}

This research was funded by the Deutsche Forschungsgemeinschaft (DFG, German Research Foundation) under Germany's Excellence Strategy – The Berlin Mathematics Research Center MATH+ (EXC-2046/1, project ID: 390685689).
\bibliographystyle{plain}
\bibliography{bibliography}
\end{document}